\documentclass[11pt,reqno,a4paper]{amsart}
\usepackage[a4paper,margin=1in]{geometry}
\usepackage[T1]{fontenc}
\usepackage[utf8]{inputenc}
\usepackage{lmodern}
\usepackage[final]{microtype}
\usepackage{amssymb}
\usepackage[dvipsnames]{xcolor}
\usepackage{enumitem}
\usepackage{amsthm}
\usepackage{dsfont}
\usepackage[colorlinks=true,linkcolor=blue,citecolor=OliveGreen,urlcolor=blue]{hyperref}
\hypersetup{
  pdftitle={An Lp-Variational Formula on Wiener Space},
  pdfauthor={David Criens, Carsten Hartmann, Michael Kupper},
  pdfsubject={Variational formula for Lp-norms of Wiener functionals},
  pdfkeywords={variational representation, Lp-norm, Wiener functional, stochastic control, stopping time, small-noise asymptotics, moment inequalities}
}
\usepackage[numbers]{natbib}
\numberwithin{equation}{section}

\makeatletter
\def\@settitle{\begin{center}%
        \baselineskip14\p@\relax
        \bfseries
        \uppercasenonmath\@title
        \@title
        \ifx\@subtitle\@empty\else
            \\[1ex]\uppercasenonmath\@subtitle
            \footnotesize\mdseries\@subtitle
        \fi
    \end{center}%
}
\def\subtitle#1{\gdef\@subtitle{#1}}
\def\@subtitle{}
\makeatother

\makeatletter
\def\@setaddresses{\par
  \nobreak \begingroup
\footnotesize
  \def\author##1{\nobreak\addvspace\bigskipamount}%
  \def\\{\unskip, \ignorespaces}%
  \interlinepenalty\@M
  \def\address##1##2{\begingroup
    \par\addvspace\bigskipamount\indent
    \@ifnotempty{##1}{(\ignorespaces##1\unskip) }%
    {\scshape\ignorespaces##2}\par\endgroup}%
  \def\curraddr##1##2{\begingroup
    \@ifnotempty{##2}{\nobreak\indent\curraddrname
      \@ifnotempty{##1}{, \ignorespaces##1\unskip}\/\space
      ##2\par}\endgroup}%
  \def\email##1##2{\begingroup
    \@ifnotempty{##2}{\nobreak\indent\emailaddrname
      \@ifnotempty{##1}{, \ignorespaces##1\unskip}\/\space
      \ttfamily##2\par}\endgroup}%
  \def\urladdr##1##2{\begingroup
    \def~{\char`\~}%
    \@ifnotempty{##2}{\nobreak\indent\urladdrname
      \@ifnotempty{##1}{, \ignorespaces##1\unskip}\/\space
      \ttfamily##2\par}\endgroup}%
  \addresses
  \endgroup
}
\makeatother

\newtheoremstyle{plainnoparen}%
{3pt plus 2pt minus 1pt}{3pt plus 2pt minus 1pt}{\itshape}{}{\bfseries}{.}{.5em}{\thmname{#1}\thmnumber{ #2}\thmnote{. #3}}

\theoremstyle{plainnoparen}
\newtheorem{theorem}{Theorem}[section]
\newtheorem{lemma}[theorem]{Lemma}

\newtheorem{corollary}[theorem]{Corollary}
\newtheorem{definition}[theorem]{Definition}
\theoremstyle{remark}
\newtheorem{remark}[theorem]{Remark}
\newtheorem{discussion}[theorem]{Discussion}

\newcommand{\1}{\mathds{1}}

\newcommand{\A}{\mathcal{A}}
\newcommand{\cF}{\mathcal{F}}
\newcommand{\cG}{\mathcal{G}}
\newcommand{\bR}{\mathbb{R}}
\renewcommand{\P}{\mathbb{P}}
\newcommand{\E}{\mathbb{E}}
\newcommand{\USA}{\textit{USA}}
\newcommand{\cE}{\mathcal{E}}
\newcommand{\Q}{\mathbb{Q}}
\newcommand{\cP}{\mathcal{P}}
\newcommand{\W}{\mathbb{W}}

\newcommand{\on}{\operatorname}

\newcommand{\Li}{\textit{Lip}_b (\Omega; \bR_+)}

\renewcommand{\epsilon}{\varepsilon}
\renewcommand{\rho}{\varrho}

\allowdisplaybreaks
\title[$L^p$-Variational Formula on Wiener Space]{An $L^p$-Variational Formula on Wiener Space}
\author[D. Criens]{David Criens}
\address{D. Criens - University of Freiburg, Ernst-Zermelo-Str. 1, 79104 Freiburg, Germany.}
\email{david.criens@stochastik.uni-freiburg.de}

\author[C. Hartmann]{Carsten Hartmann}
\address{C. Hartmann - Brandenburg University of Technology Cottbus--Senftenberg, Institute of Mathematics, Konrad-Wachsmann-Allee 1, 03046 Cottbus, Germany.}
\email{carsten.hartmann@b-tu.de}

\author[M. Kupper]{Michael Kupper}
\address{M. Kupper - University of Konstanz, 78457 Konstanz, Germany.}
\email{kupper@uni-konstanz.de}

\date{}

\begin{document}

\begin{abstract}
We establish a stochastic control representation for $L^p$-norms on
Wiener space. For every $p\ge1$ and every non-negative universally
measurable functional $\varphi$, we show that
\[
\|\varphi(W)\|_p
=
\sup_a
\E\Big[
e^{-\frac12\int_0^T\|a_t\|^2\,dt}
\varphi\Big(
W+\sqrt{p-1}\int_0^\cdot a_t\,dt
\Big)
\Big],
\]
where $W$ is a $d$-dimensional Brownian motion and the supremum is taken over
all progressively measurable controls $a$ satisfying
$\int_0^T\|a_t\|^2\,dt<\infty$ almost surely. This identity can be viewed as a multiplicative $L^p$-analogue of the Bou\'e--Dupuis variational formula. We also provide an extension to strong solutions of stochastic differential equations with
path-dependent coefficients.
Building on the variational formula, we discuss finite-\(p\) estimates for small-noise limits of stochastic differential equations with path-dependent coefficients. These include quantitative
Freidlin--Wentzell bounds, a functional Gaussian approximation and concentration estimates, and
moderate-deviation bounds. Moreover, we discuss consequences for functionals determined up to a stopping time, obtaining in particular probability bounds that identify the sharp exponential
decay rate after optimization over
controls.
The proof of the variational formula combines probabilistic arguments with convex duality methods, viscosity theory for Hamilton--Jacobi--Bellman equations, and a dynamic programming principle.


\medskip
\noindent\emph{MSC2020 subject classifications.} Primary 60F10, 60J65,
93E20; secondary 49L20, 60G40, 60H10.
\end{abstract}

\maketitle

\section{Introduction}
Variational formulas provide a useful link between probabilistic quantities and stochastic control problems, and have found applications in several areas of probability. A classical example is the Bou{\'e}--Dupuis variational formula, which gives a control representation for logarithmic exponential moments and is widely used in small-noise analysis, risk-sensitive control, and related exponential estimates. In this paper, we establish a variational formula of a different type, giving a direct stochastic control representation for \(L^p\)-norms of Brownian functionals.

Let $\Omega:=C([0,T];\bR^d)$ be the path space, let $X_t(\omega):=\omega (t)$ denote the coordinate process, and let $\W$ be the Wiener measure on $\Omega$.
 We denote by $\A$ the set of all $\bR^d$-valued processes that are progressively measurable with respect to the \(\W\)-completed right-continuous natural filtration of $X$ and satisfy $\int_0^T \|a_t\|^2\,dt<\infty$, $\W$-a.s.
 Our main result states that,
for every $p\ge1$ and every non-negative universally measurable
functional $\varphi\colon\Omega\to[0,\infty]$,
\begin{equation}\label{eq-CK-intro}
\E^\W\Big[\varphi(X)^p\Big]^{1/p}
=
\sup_{a\in\A}
\E^\W\Big[
e^{
-\frac12\int_0^T\|a_t\|^2\,dt
}
\varphi\Big(
X+\sqrt{p-1}\int_0^\cdot a_t\,dt
\Big)
\Big],
\end{equation}
where both sides may take the value $+\infty$.
For $p=1$, equation~\eqref{eq-CK-intro} is immediate, since the shift
vanishes, the exponential factor is at most one, and the zero control
$a\equiv0$ attains the supremum. 

The natural point of comparison is Borell's variational formula for
functions of a Gaussian vector \cite{Borell2000} and the corresponding
Wiener-space formula of Bou\'e and Dupuis
\cite{BoueDupuis1998}. In the latter formulation, every bounded
measurable functional $\psi$ satisfies
\[
\log\E^\W \big[\exp(\psi(X)) \big]
=
\sup_{a\in\A}
\E^\W \Big[
\psi\Big(X+\int_0^\cdot a_t\,dt\Big)
-\frac12\int_0^T\|a_t\|^2\,dt
\Big].
\]
If $\varphi$ is bounded and bounded away from zero, applying this
formula with $\psi=p\log\varphi$ and dividing by $p$ yields a representation for
$\log\|\varphi(X)\|_p$. Exponentiating the resulting identity does not
yield \eqref{eq-CK-intro}, because the exponential remains outside the
controlled expectation. Equation~\eqref{eq-CK-intro} instead represents the $L^p$-norm
directly as the supremum over all controls of the expectation of
the shifted functional multiplied by the exponential control factor.
An advantage of \eqref{eq-CK-intro} is that it remains meaningful when
$\varphi$ takes the value zero. To obtain the corresponding
Bou\'e-Dupuis representation, one applies that formula to
$p\log\varphi$, so zeros of $\varphi$ must be handled by approximation.
No such approximation is needed in \eqref{eq-CK-intro}, where
$\varphi$ itself appears. 
In particular, for every universally
measurable set $B\subseteq\Omega$, taking $\varphi=\1_B$ gives a
control representation of the powered probability $\W(B)^{1/p}$.

A second comparison comes from power duality, which represents the
$L^p$-norm over absolutely continuous probability measures
\cite[Proposition~2.5]{PMR_96}. For every $\Q\ll\W$ with finite relative
entropy, the coordinate process under
$\Q$ decomposes into a Brownian motion and a drift, and the expected
integral of its squared norm equals twice the relative entropy
\cite{Foellmer1985,Foellmer1986,L_13, LS_01}. This connects changes of measure
with drifts but does not directly yield \eqref{eq-CK-intro},
since the drift and its quadratic integral are evaluated under $\Q$,
whereas the shifted functional and the exponential cost factor in
\eqref{eq-CK-intro} are evaluated under $\W$. Further details are given
in Discussion~\ref{diss: main result}.

\subsubsection*{Proof sketch}
At an intuitive level, the proof of \eqref{eq-CK-intro} is split into two parts, dealing with the inequalities \(\leq\) and \(\geq\), respectively. The inequality \(\geq\) can be established directly from a local change of measure and an application of H\"older's inequality, see Lemma~\ref{lem: geq inequality in main}. The converse inequality \(\leq\) requires more preparation. Our proof strategy combines tools related to convex expectations on path space, viscosity theory for Hamilton--Jacobi--Bellman PDEs, and stochastic optimal control. 
To outline the basic idea, we consider both sides of \eqref{eq-CK-intro} as convex expectations on the path space \(\Omega\), i.e., as convex monotone functionals that preserve non-negative constants and satisfy certain continuity from above and below properties. Using their dual representation from \cite{BartlCheriditoKupper2019}, we show that convex expectations are ordered for all path-dependent functionals whenever their finite-dimensional distributions are ordered. This extends a comparison principle from~\cite{CriensKupper2025} to a framework of convex expectations that only preserve non-negative constants, which is tailor-made to treat settings with multiplicative penalization. As a consequence of our comparison result, the inequality \(\leq\) from \eqref{eq-CK-intro} holds for arbitrary non-negative \(\varphi\) once it holds for all functions of the type \(\varphi = f (X_{t_1}, \dots, X_{t_n})\) with \(\inf f > 0\). 
For these we can prove \eqref{eq-CK-intro} by PDE methods. Namely, we show that both sides are related to viscosity solutions to the same non-linear backward Hamilton--Jacobi--Bellman PDE that has a suitable uniqueness property to conclude \eqref{eq-CK-intro}. More detailed explanations of this strategy are given in Discussion~\ref{diss: main result} and Section~\ref{sec: pf VF}.

The variational formula \eqref{eq-CK-intro} extends to strong solutions of stochastic
differential equations with path-dependent coefficients. For $p>1$ and $a\in\A$, let $Y^{(p)}$ and $Z^{(p,a)}$ be the
solutions of
\begin{equation}\label{eq-intro-controlled-SDEs}
\begin{aligned}
dY_t^{(p)}
&=
\mu(t,Y^{(p)})\,dt
+\frac1{\sqrt{p-1}}\sigma(t,Y^{(p)})\,dX_t,\\
dZ_t^{(p,a)}
&=
\bigl(\mu(t,Z^{(p,a)})+\sigma(t,Z^{(p,a)})a_t\bigr)\,dt
+\frac1{\sqrt{p-1}}\sigma(t,Z^{(p,a)})\,dX_t,
\end{aligned}
\end{equation}
with common initial condition
$Y_0^{(p)}=Z_0^{(p,a)}=y_0$. Under nonanticipativity, linear-growth, and local Lipschitz conditions
specified in Subsection~\ref{sec-VF-SDE}, we obtain in
Theorem~\ref{theo-VF-SDE} that
\begin{equation}\label{eq-intro-sde-variational}
\E^\W\Big[\varphi(Y^{(p)})^p\Big]^{1/p}
=
\sup_{a\in\A}
\E^\W\Big[
e^{-\frac12\int_0^T\|a_t\|^2\,dt}
\varphi(Z^{(p,a)})
\Big].
\end{equation}
The normalization in \eqref{eq-intro-controlled-SDEs} links the moment
order $p$ to the size of the Brownian term. In both $Y^{(p)}$ and
$Z^{(p,a)}$, the diffusion coefficient is
$(p-1)^{-1/2}\sigma$, whereas the additional drift $\sigma a$ is
independent of $p$. 

Section~\ref{sec-applications} develops three
applications of these variational formulas, related to quantitative small-noise large and moderate deviation estimates. 
The quantitative comparisons in Theorem~\ref{thm-app-finite-p-comparison}
and Theorem~~\ref{thm-app-moderate-comparison} are further main results.
They give explicit finite-$p$ error bounds for deterministic and Gaussian
approximations, rather than only identifying asymptotic limits.
We now discuss these applications in more detail.

\subsubsection*{Small-noise approximations}
First, we use the stochastic control problem in \eqref{eq-intro-sde-variational} to establish a quantitative comparison with its deterministic
counterpart. For a deterministic control $b\in L^2([0,T];\bR^d)$, let
$z^{(b)}$ be a solution to 
\[
 dz_t^{(b)}=(\mu(t,z^{(b)})+\sigma(t,z^{(b)})b_t)\,dt, \quad z_0^{(b)}=y_0. 
\]
We define the associated rate function by
\[
I(\omega)
:=
\inf\Big\{
\frac12\int_0^T\|b_t\|^2\,dt
\mathrel{\colon}
b\in L^2([0,T];\bR^d),\ z^{(b)}=\omega
\Big\},
\]
and set $I(B):=\inf_{\omega\in B}I(\omega)$ for
$B\subseteq\Omega$. Thus, $I(\omega)$ is the minimal cost of a deterministic control satisfying $z^{(b)}=\omega$. Under
global Lipschitz assumptions on the coefficients,
Theorem~\ref{thm-app-finite-p-comparison} shows that $I(\, \cdot \,)$ is a
good rate function and that there is a constant $K\ge1$, independent
of $p$ and $\varphi$, such that, for every $p\ge2$ and every
$\varphi\colon\Omega\to[0,\infty)$ with concave modulus of
continuity $\varpi$,
\begin{equation}\label{eq-intro-finite-p-comparison}
\Big|
\E^\W\Big[\varphi(Y^{(p)})^p\Big]^{1/p}
-
\max_{\omega\in\Omega}
e^{-I(\omega)}\varphi(\omega)
\Big|
\le
\varpi\Big(\frac{K}{\sqrt{p-1}}\Big).
\end{equation}
The maximum is finite and attained. Its maximizing path balances the
value of $\varphi$ against the smallest cost required to produce that
path.

For $\varphi=e^\psi$, the deterministic value in
\eqref{eq-intro-finite-p-comparison} is
$\exp (\sup_{\omega\in\Omega}(\psi(\omega)-I(\omega)) )$, whose
exponent is the familiar value from the Laplace principle in the
theory of large deviations.
Classical small-noise theory identifies this limit for Markov
diffusions \cite{FreidlinWentzell2012}, with extensions to
path-dependent diffusions in
\cite{ChiariniFischer2014,MaRenTouziZhang2016}. Variational
weak-convergence methods for large deviations are developed in
\cite{BudhirajaDupuis2000,BudhirajaDupuis2019,DupuisEllis1997}, while
nonexponential versions of Sanov's and Schilder's theorems based on
convex duality appear in
\cite{BackhoffLackerTangpi2020,Lacker2020Sanov}. These asymptotic
results identify the limiting value, whereas
\eqref{eq-intro-finite-p-comparison} provides a quantitative comparison
at every finite $p$. For an $\alpha$-H\"older functional, its explicit
error is of order $(p-1)^{-\alpha/2}$. Lipschitz approximations of
indicators further yield two-sided finite-$p$ probability bounds in
terms of inner and outer neighborhoods. 

\subsubsection*{Moderate deviations}
We next study fluctuations around the deterministic zero-control path
$z^{(0)}$, which is related to the moderate deviation regime. The Brownian term in \eqref{eq-intro-controlled-SDEs} is
multiplied by $1/\sqrt{p-1}$, while $z^{(0)}$ solves the corresponding
equation without this term. This motivates rescaling
$Y^{(p)}-z^{(0)}$ by $\sqrt{p-1}$. Under the additional assumption that
the drift admits a linear expansion along $z^{(0)}$ with a quadratic
remainder, we denote its linear part by $A_t$. Linearizing the drift and
freezing the diffusion coefficient along $z^{(0)}$ then lead to
\[
\sqrt{p-1}\bigl(Y^{(p)}-z^{(0)}\bigr)
\approx G,
\qquad
G_t
=
\int_0^t A_sG\,ds
+
\int_0^t\sigma(s,z^{(0)})\,dX_s.
\]
Since the coefficients in this linear equation are deterministic, $G$
is a centered Gaussian process.

The corresponding deterministic equation is obtained in the same
way. If the deterministic path $z^{(0)}$ is perturbed by a small
control $\varepsilon b$, then its first-order displacement is
$v^{(b)}$, where
\[
v_t^{(b)}
=
\int_0^t A_sv^{(b)}\,ds
+
\int_0^t\sigma(s,z^{(0)})b_s\,ds,
\qquad
b\in L^2([0,T];\bR^d).
\]
More precisely,
$(z^{(\varepsilon b)}-z^{(0)})/\varepsilon$ converges to
$v^{(b)}$ as $\varepsilon\downarrow0$. Hence, $G$ describes the
first-order random fluctuation around $z^{(0)}$, while $v^{(b)}$
describes the first-order displacement produced by the control $b$.
As in the definition of $I(\, \cdot \,)$, let
$I_{\mathrm{lin}}(\eta)$ be the infimum of
$\frac12\int_0^T\|b_t\|^2\,dt$ over all controls $b$ satisfying
$v^{(b)}=\eta$.
In Theorem~\ref{thm-app-moderate-comparison}, we show that the expected
uniform distance, with the exponential factor from the variational
representation, is bounded by $Kh/\sqrt{p-1}$ uniformly over all
controls. Combining this estimate with the preceding comparison for the
linearized equation shows that, for every $p\ge2$,
$1\le h\le\sqrt{p-1}$, and every non-negative functional $\varphi$ with
concave modulus of continuity $\varpi$,
\begin{equation}\label{eq-intro-moderate-comparison}
\Big|
\E^\W\Big[
\varphi\Big(
\frac{\sqrt{p-1}}{h}
\bigl(Y^{(p)}-z^{(0)}\bigr)
\Big)^{1+h^2}
\Big]^{1/(1+h^2)}
-
\max_{\eta\in\Omega}
e^{-I_{\mathrm{lin}}(\eta)}\varphi(\eta)
\Big|
\le
\varpi\Big(
K\Big(
\frac1h+\frac h{\sqrt{p-1}}
\Big)
\Big).
\end{equation}
The term $1/h$ is the error in comparing the Gaussian linearized
equation with its deterministic control problem, while
$h/\sqrt{p-1}$ is the error in approximating the original controlled
equation by its linearization. The estimate between the controlled
paths also yields a direct Gaussian approximation: for
every $p\ge2$ and $2\le q\le p$,
\[
\E^\W\Big[
\Big\|
\sqrt{p-1}\bigl(Y^{(p)}-z^{(0)}\bigr)-G
\Big\|_T^q
\Big]^{1/q}
\le
\frac{K(q-1)}{\sqrt{p-1}}.
\]
For every fixed $q$, this is a quantitative functional central limit
theorem in the uniform topology. Moreover, Corollary~\ref{cor:finite-p-concentration} shows that, for every
$p\ge2$, $1\le u\le p-1$, and every $L$-Lipschitz functional
$\varphi\colon\Omega\to\bR$ with $L>0$,
\[
\W\Big(
\varphi(Y^{(p)})-\E^\W[\varphi(Y^{(p)})]
\ge
KL\Big(
\sqrt{\frac{u}{p-1}}+\frac{u}{p-1}
\Big)
\Big)
\le
2e^{-u},
\]
and the same estimate holds for the lower tail. The square-root term
comes from Gaussian concentration, while the second term comes from
the error between the diffusion and its Gaussian approximation.

If $h=h_p\to\infty$ and
$h_p/\sqrt{p-1}\to0$, both terms in
\eqref{eq-intro-moderate-comparison} vanish.
Corollary~\ref{cor-app-moderate-deviations} provides explicit pre-limit upper and
lower bounds for the probability that the
rescaled process belongs to a prescribed set. In this regime,
these bounds establish the moderate-deviation principle with speed
$h_p^2$ and rate function $I_{\mathrm{lin}}(\, \cdot \,)$. When
$h=\sqrt{p-1}$, the rescaled process is $Y^{(p)}-z^{(0)}$, but the
linearization error no longer vanishes. The nonlinear comparison
\eqref{eq-intro-finite-p-comparison} is then the relevant estimate.
 
Gaussian approximations for Markov diffusions with additive noise
are quantified in relative entropy on path space in
\cite{SanzAlonsoStuart2017}, whereas variational and pathwise
moderate-deviation results are developed in
\cite[Chapters~9-10]{BudhirajaDupuis2019} and
\cite{JacquierSpiliopoulos2020,SuoTaoZhang2018}. The comparison
\eqref{eq-intro-moderate-comparison} holds on path space for
path-dependent coefficients and every finite $p$ and $h$. It keeps
the Gaussian and linearization errors explicit, and both the
quantitative Gaussian approximation and the moderate-deviation
principle follow from the underlying controlled path estimates.

\subsubsection*{Extensions}
Finally, we turn to functionals determined by the path up to a stopping
time. 
For a raw stopping time \(\tau\) (e.g.,~a first hitting time of a closed set) and a non-negative functional
$\varphi\colon\Omega\to[0,\infty]$, measurable with respect to the \(\sigma\)-field
$\sigma(X_{t\wedge\tau},\,t\in[0,T])$, Theorem~\ref{thm-app-stopped-functional} shows,
for $p>1$,
\begin{equation}\label{eq-intro-stopped-variational}
\E^\W\Big[\varphi(Y^{(p)})^p\Big]^{1/p}
=
\sup_{a\in\A}
\E^\W\Big[
e^{
-\frac12
\int_0^{\tau(Z^{(p,a)})}\|a_t\|^2\,dt
}
\varphi(Z^{(p,a)})
\Big].
\end{equation}
Applied to indicators, the stopped representation yields bounds for the
probability that the controlled path belongs to $B \subseteq \Omega$ and that its control
cost up to the stopping time does not exceed a given level $R$. Combining these bounds
with the deterministic comparison
\eqref{eq-intro-finite-p-comparison} identifies the exact asymptotic
decay rate after optimization over the controls. More precisely, if
$B\in\sigma (X_{t \wedge \tau}, t \in [0, T])$ satisfies
$I(B^\circ)=I(B)=I(\overline B)<\infty$, then, for every $R\ge0$,
\begin{equation}\label{eq-intro-two-scales}
\begin{aligned}
&\sup_{a\in\A}
\W\Big(
Z^{(p,a)}\in B,\;
\frac12
\int_0^{\tau(Z^{(p,a)})}\|a_t\|^2\,dt
\le R
\Big)\\
&\quad\qquad=
\exp\Big(
-(p-1)
\Big[
\Big(\sqrt{I(B)}-\sqrt R\Big)_+^2+o(1)
\Big]
\Big)
\qquad\text{as }p\to\infty,
\end{aligned}
\end{equation}
where $x_+:=\max\{x,0\}$. Equation~\eqref{eq-intro-two-scales} shows
how the bound $R$ changes the asymptotic decay rate. For $R=0$, the rate
is $I(B)$. It decreases as $R$ increases and vanishes once
$R\ge I(B)$, in which case the optimized probability is no longer
exponentially small on the scale $p-1$.

Risk-sensitive control of exits for Markov diffusions is studied in
\cite{BoueDupuis2001,DupuisMcEneaney1997}, while a logarithmic
small-noise limit for a path-dependent exit problem appears in
\cite{CriensFuchs2026}. The stopped identity
\eqref{eq-intro-stopped-variational} is exact for every $p>1$ and
applies to any non-negative functional determined at a stopping time,
without requiring a Markov structure or a representation as an exit
from a state-space domain.

\medskip
The paper is organized as follows.
Section~\ref{sec-variational-formulas} states the variational formula
and its extension to path-dependent stochastic differential equations.
Section~\ref{sec-applications} develops the applications, and
Section~\ref{sec: pf VF} proves the main representation. The appendix
contains the dynamic programming principle used in the
finite-dimensional reduction.

\section{Variational formulas for Brownian functionals}
\label{sec-variational-formulas}

We first state the variational formula on canonical Wiener space and then extend it to strong solutions of path-dependent stochastic differential equations.  The proof of the main formula is given in Section~\ref{sec: pf VF}.

\subsection{Setup} \label{sec: setup}

Fix a finite time horizon $T>0$ and a dimension $d\in\mathbb N$.  Set $\Omega := C([0,T];\bR^d)$, endow this space with the uniform topology, and denote the corresponding Borel \(\sigma\)-field by $\cF := \mathcal B(\Omega)$. The coordinate process on \(\Omega\) is denoted by $X=(X_t)_{t\in[0,T]}$, i.e., $X_t(\omega)=\omega(t)$ for \(\omega \in \Omega\) and \(t \in [0, T]\). We denote the Wiener measure on \((\Omega, \cF)\) by $\mathbb{W}$. 
Let $(\cF_t^0)_{t\in[0,T]}$ be the raw natural filtration generated by
$X$, given by $\cF_t^0:=\sigma(X_s,\ s\in[0,t])$, and let
$(\cF_t)_{t\in[0,T]}$ be its right-continuous version, where
$\cF_t:=\bigcap_{r\in(t,T]}\cF_r^0$ for $t<T$ and $\cF_T:=\cF$.
Denote
by $\cF^\W$ and $(\cF_t^\W)_{t\in[0,T]}$ the respective
$\W$-completions of $\cF$ and $(\cF_t)_{t\in[0,T]}$. Unless another
filtered probability space is specified, we work on
\[
(\Omega,\cF^\W,(\cF_t^\W)_{t\in[0,T]},\W).
\]
Throughout, $\|\cdot\|$ denotes the Euclidean norm on $\bR^d$ and the
Hilbert-Schmidt norm on $\bR^{d\times d}$.
Let $\A$ denote the set of all $\bR^d$-valued progressively measurable
processes $a=(a_t)_{t\in[0,T]}$ which satisfy
$\int_0^T\|a_t\|^2\,dt<\infty$, $\W$-a.s.
For every $a\in\A$, the process
$t\mapsto\int_0^t a_s\,ds$ admits a continuous progressively measurable
version, which is used in the path shifts below.
Finally, for two Polish spaces $E$ and $F$, recall that a function
$\varphi\colon E\to F$ is universally measurable if it is
$\mathcal B^\mu(E)$-$\mathcal B(F)$-measurable for every probability
measure $\mu$ on $(E,\mathcal B(E))$, where $\mathcal B^\mu(E)$ denotes
the $\mu$-completion of $\mathcal B(E)$. Although not obvious from the definition, the composition of two universally measurable functions is again universally measurable, see \cite[Proposition~7.44]{bershre}. 

\subsection{A variational formula for Brownian functionals} \label{sec: VF}

The following theorem is the main result of the paper.  Its proof is postponed to Section~\ref{sec: pf VF}.

\begin{theorem}\label{theo-VF}
For every universally measurable function $\varphi\colon\Omega\to[0,\infty]$ and every $p\ge1$,
\begin{equation}\label{eq-main-VF}
\E^\W \big[ \varphi (X)^p \big]^{1/p}
=
\sup_{a\in\A}
\E^\W
\Big[ e^{-\frac{1}{2}\int_0^T\|a_s\|^2\,ds}\varphi\Big( X + \sqrt{p-1}\,\int_0^\cdot a_s \, ds \Big)
\Big].
\end{equation}
Both sides may take the value $+\infty$.
\end{theorem}

\label{theo: VF}

In the following discussion we sketch the main ideas of our proof for Theorem~\ref{theo-VF}, relate it to the famous Bou\'e--Dupuis \cite{BoueDupuis1998} formula, and comment on alternative proof strategies. 

\begin{discussion} \label{diss: main result}
Our proof of Theorem~\ref{theo-VF} combines a fixed-control estimate
with convex duality arguments and PDE techniques. Localization,
Girsanov's theorem, and H\"older's inequality show that the control
value does not exceed the $L^p$-norm. The same estimate provides the
continuity from above needed for the subsequent duality argument.

For the converse inequality, we consider the left- and right-hand sides
of \eqref{eq-main-VF} as convex expectations over upper semianalytic
functions on $\Omega$. Recall that a function
$\varphi\colon\Omega\to\bR$ is upper semianalytic if
$\{\varphi\ge t\}$ is analytic for every $t\in\bR$. Every Borel
function is upper semianalytic. Furthermore, upper semianalytic
functions are universally measurable, see
\cite[Section~7.7]{bershre}.

The argument proceeds in three steps. First, we use convex duality to
reduce the remaining inequality to functionals of the form
$\varphi(X_{t_1},\ldots,X_{t_n})$, where $n\in\mathbb N$,
$0\le t_1<\cdots<t_n\le T$, and
$\varphi\in\textit{Lip}_b(\bR^{dn})$. 
In the second step, using the Markov property of Brownian motion,
dynamic programming, and properties of convex expectations, we reduce
this identity further to functionals of the form $\varphi(X_t)$, where
$t\in[0,T]$ and $\varphi\in\textit{Lip}_b(\bR^d)$ satisfies
$\inf\varphi>0$.
Consequently, the variational formula \eqref{eq-main-VF} holds for all
non-negative upper semianalytic $\varphi$ once
\begin{align}\label{eq: main to show}
\E^\W\Big[\varphi(X_t)^p\Big]^{1/p}
=
\sup_{a\in\A}
\E^\W\Big[
e^{-\frac12\int_0^T\|a_s\|^2\,ds}
\varphi\Big(
X_t+\sqrt{p-1}\int_0^t a_s\,ds
\Big)
\Big]
\end{align}
is established for all $t\in[0,T]$ and
$\varphi\in\textit{Lip}_b(\bR^d)$ with $\inf\varphi>0$.

To prove this, we use a PDE argument. Specifically, for
$\varphi\in\textit{Lip}_b(\bR^d)$ bounded away from zero and
$(t,x)\in[0,T]\times\bR^d$, set
\begin{align*}
u(t,x)
&:=
\E^\W\big[\varphi(x+X_{T-t})^p\big],\\
v(t,x)
&:=
\sup_{a\in\A}
\E^\W\Big[
e^{-\frac12\int_0^T\|a_s\|^2\,ds}
\varphi\Big(
x+X_{T-t}+\sqrt{p-1}\int_0^{T-t}a_s\,ds
\Big)
\Big].
\end{align*}
It is classical that $u$ solves the backward heat equation. Moreover,
by standard change of variable rules,
$\widetilde v:=u^{1/p}$ is a bounded viscosity solution of the
Hamilton--Jacobi--Bellman (HJB) equation
\[
-\frac{\partial\widetilde v}{\partial t}
-\sup_{a\in\bR^d}
\Big[
\sqrt{p-1}\langle D\widetilde v,a\rangle
-\frac12\widetilde v\|a\|^2
\Big]
-\frac12\on{tr}\big[D^2\widetilde v\big]
=
0.
\]
By dynamic programming, $v$ solves the same HJB equation with the same
terminal condition. Since uniqueness for this equation follows from
uniqueness for the backward heat equation, we conclude that
$\widetilde v=v$, which proves \eqref{eq: main to show} for every
$\varphi\in\textit{Lip}_b(\bR^d)$ bounded away from zero. Approximating
$\varphi$ by $\varphi+\delta$ and letting $\delta\downarrow0$ removes
this restriction. The backward recursion and the comparison principle
then yield the converse inequality for all non-negative bounded upper
semianalytic functionals. Together with the fixed-control estimate,
this proves \eqref{eq-main-VF}. Finally, Borel versions, absolute
continuity of the controlled shifted laws, and monotone convergence
extend the formula to all non-negative universally measurable
functionals.

  This strategy differs from the classical proof of the Bou\'e--Dupuis \cite{BoueDupuis1998} formula, which states, for bounded real-valued Borel \(\varphi\), that
	\begin{align} \label{eq: BD formula}
		\log \E^{\W} \big[ e^{ \varphi (X)} \big] = \sup_{a \in \A} \E^{\W} \Big[ \varphi \Big( X + \int_0^\cdot a_s \, ds \Big) - \frac{1}{2} \int_0^T \| a_s \|^2 \, ds \Big].
	\end{align}
	Its proof relies on the dual representation of the entropic risk measure, given by
	\begin{align} \label{eq: dual rep BD formula}
		\log \E^{\W} \big[ e^{\varphi (X)}\big] = \sup_{\substack{ \Q \, \in \, \mathcal{P}(\Omega) \\ \Q \ll \W }} \Big( \E^{\Q} \big[ \varphi (X)\big] - H (\Q \mid \W) \Big), 
	\end{align}
	where \(H (\Q \mid \W) := \E^{\Q} [ \log ( \frac{d\Q}{d\W})]\) denotes the relative entropy. On an intuitive level, the formulas \eqref{eq: BD formula} and
\eqref{eq: dual rep BD formula} are related by the fact that every
$\Q\ll\W$ with $H(\Q\mid\W)<\infty$ corresponds to a Brownian motion
with an absolutely continuous drift $\int_0^\cdot a_s\,ds$. In this
case, it follows from \cite[Theorem~2]{L_13} that
\[
H(\Q\mid\W)
=
\frac12
\E^\Q\Big[
\int_0^T\|a_s\|^2\,ds
\Big].
\]

The $L^p$-norm also admits a representation through changes of measure. Namely, \cite[Proposition~2.5]{PMR_96} shows that, for non-negative \(\varphi\) and \(1 \leq p < \infty\),
	\begin{align} \label{eq: variational norm}
		\E^{\W} \big[ \varphi (X)^p \big]^{1 / p} = \sup_{\substack{ \Q \, \in \, \mathcal{P}(\Omega) \\ \Q \ll \W }} \E^{\Q} \Big[ \varphi (X) \, \Big( \frac{d \Q}{d \W} \Big)^{- \frac{1}{p}} \Big].
	\end{align}
A related change of measure representation follows from backward
stochastic differential equations (BSDEs). For $p>1$ and a bounded
Borel functional $\varphi$ with $\inf\varphi>0$, the process
$V_t:=\E^\W[\varphi(X)^p\mid\cF_t^\W]^{1/p}$ solves a BSDE with
terminal value $\varphi(X)$ and generator
$(y,z)\mapsto(p-1)\|z\|^2/(2y)$ for $y>0$.
A general framework for BSDE dual representations with discounting
and changes of measure is developed in
\cite{DrapeauKupperRosazzaGianinTangpi2016}.
For this generator, the weak representation of the $L^p$-norm is
\begin{equation}\label{eq-bsde-discounted-representation}
\E^\W\Big[\varphi(X)^p\Big]^{1/p}
=
\sup_{\substack{q\in\A \\ q\text{ bounded}}}
\E^{\Q^q}\Big[
e^{-\frac1{2(p-1)}\int_0^T\|q_t\|^2\,dt}
\varphi(X)
\Big],
\end{equation}
where $\frac{d\Q^q}{d\W}
=\exp(\int_0^T q_t\,dX_t-\frac12\int_0^T\|q_t\|^2\,dt)$.
The strong formulation \eqref{eq-main-VF} requires additional arguments,
since the functional is evaluated at adapted shifts of the Brownian
path under the original Wiener measure.
For additive costs, weak representations under changes of measure are
obtained through BSDE methods in
\cite[Remark~3.5]{DelbaenHuRichou2011} and
\cite[Theorem~3.2]{DelbaenPengRosazzaGianin2010}.
For a class of deterministic convex cost functions,
\cite[Theorem~3.1]{BackhoffLackerTangpi2020} establishes the
corresponding strong representation for bounded Borel functionals.
An alternative proof of \eqref{eq-main-VF} could approximate
the Girsanov drifts by bounded elementary controls, following the proof
of \cite[Lemma~3.2]{BackhoffLackerTangpi2020}, for which the discounted
expectations in \eqref{eq-bsde-discounted-representation} can be
rewritten using adapted shifts. Passing to the limit would require
convergence of these expectations and an extension to all non-negative
universally measurable functionals. Here, we develop the
proof through a general comparison principle for convex expectations
on path space, which is of independent interest, and combine it with
the finite-dimensional PDE argument described above.
    
Lastly, we comment on possible direct PDE approaches to
\eqref{eq-main-VF}. As outlined above, our proof uses a PDE argument to
establish the formula for functionals of the form $\varphi(X_t)$. This
argument does not directly extend to time-integrated functionals. More
specifically, for
$\varphi=\int_0^T g(X_s)\,ds$, the right-hand side of
\eqref{eq-main-VF} takes the form
\[
\sup_{a\in\A}
\E^\W\Big[
e^{-\frac12\int_0^T\|a_s\|^2\,ds}
\int_0^T
g\Big(
X_s+\sqrt{p-1}\int_0^s a_r\,dr
\Big)ds
\Big].
\]
Here, the exponential factor multiplies the entire integral, rather
than discounting each contribution up to its integration time.
A direct PDE treatment would therefore have to account for the
accumulated integral as an additional state variable.
Our approach avoids this state enlargement by reducing the
identification to functionals of the form $\varphi(X_t)$.
\end{discussion}

Although Theorem~\ref{theo-VF} is stated on the canonical space, its value does not depend on this particular realization of Brownian motion.  The following formulation also allows additional randomness within the filtration.

\begin{corollary}\label{coro-arbitrary-setup}
Let $(\widetilde{\Omega},\widetilde{\cF},(\widetilde{\cF}_t)_{t \in [0, T]},\widetilde{\P})$ be a filtered probability space with complete filtration that supports a \(d\)-dimensional $(\widetilde\cF_t)_{t \in [0, T]}$-Brownian motion $\widetilde{W}$. Let $\widetilde{\A}$ be the set of all $\bR^d$-valued $(\widetilde{\cF}_t)_{t \in [0, T]}$-progressively measurable processes \((a_t)_{t \in [0, T]}\) satisfying  $\int_0^T\|a_s\|^2\,ds<\infty$, \(\widetilde{\P}\)-a.s. Then, for every non-negative universally measurable function $\varphi \colon \Omega \to \bR_+$ and every $p\ge1$,
\[
\E^{\widetilde{\P}} \Big[ \varphi \big(\widetilde{W}\big)^p \Big]^{1/p}
=
\sup_{a\in\widetilde\A}
\E^{\widetilde\P}\Big[
e^{-\frac{1}{2}\int_0^T\|a_s\|^2\,ds}
\varphi\Big(\widetilde W+\sqrt{p-1}\int_0^\cdot a_s\,ds\Big)
\Big].
\]
\end{corollary}

\begin{proof}
By virtue of \cite[Lemma~2.1]{STZ_11}, for each \(a \in \mathcal{A}\), there exists an \((\cF^0_t)_{t \in [0, T]}\)-progressively measurable control process \(a'\) such that \(a = a'\)
 $dt\otimes\W$-almost everywhere. 
 Using this notation and the fact that \(a' (\widetilde{W})\in \widetilde{\mathcal{A}}\) for every \(a \in \mathcal{A}\), we obtain that 
 \begin{align*}
     \E^{\widetilde{\P}} \Big[ \varphi \big(\widetilde{W}\big)^p \Big]^{1/p} &= \E^\W \Big[ \varphi (X)^p \Big]^{1/p}
     \\&= \sup_{a \in \mathcal{A}} \E^\W \Big[ e^{- \frac{1}{2} \int_0^T \| a'_s \|^2 \, ds} \varphi \Big( X + \sqrt{p - 1} \int_0^\cdot a'_s \, ds \Big) \Big] 
     \\&= \sup_{a \in \mathcal{A}} \E^{\widetilde{\P}} \Big[ e^{- \frac{1}{2} \int_0^T \| a'_s (\widetilde{W}) \|^2 \, ds} \varphi \Big( \widetilde{W} + \sqrt{p - 1} \int_0^\cdot a'_s (\widetilde{W}) \, ds \Big) \Big]
     \\&\leq \sup_{a\in\widetilde\A}
\E^{\widetilde\P}\Big[
e^{-\frac{1}{2}\int_0^T\|a_s\|^2\,ds}
\varphi\Big(\widetilde W+\sqrt{p-1}\int_0^\cdot a_s\,ds\Big)
\Big].
 \end{align*}
 The converse inequality \(\geq\) is a consequence of Lemma~\ref{lem: geq inequality in main} below. 
\end{proof}

\subsection{Extension to path-dependent stochastic differential equations}
\label{sec-VF-SDE}

We now extend Theorem~\ref{theo-VF} to functionals of strong solutions of path-dependent stochastic differential equations.  Let $\mu\colon[0,T]\times\Omega\to\bR^d$ and $\sigma\colon[0,T]\times\Omega\to\bR^{d\times d}$ be two coefficients, and fix an initial value~$y_0\in\bR^d$.  For $\omega\in\Omega$ and \(t \in [0, T]\), set $\|\omega\|_t := \sup_{0\le s\le t}\|\omega(s)\|$ and $\omega_{\cdot\wedge t}(s) :=\omega(s\wedge t)$ for \(s \in [0, T]\).  We make the following assumptions.
\begin{enumerate}[label=\textup{(A\arabic*)}]
\item The coefficients are Borel measurable and nonanticipative in the sense that, for every $t\in[0,T]$ and $\omega\in\Omega$,
\[
\mu(t,\omega)=\mu(t,\omega_{\cdot\wedge t}),
\qquad
\sigma(t,\omega)=\sigma(t,\omega_{\cdot\wedge t}).
\]
\label{A1}
\item There is a constant $C>0$ such that
\[
\|\mu(t,\omega)\|+\|\sigma(t,\omega)\|
\le
C(1+\|\omega\|_t)
\]
for every $t\in[0,T]$ and $\omega\in\Omega$.
\label{A2}
\item For every $N>0$, there is a constant $L_N>0$ such that
\[
\|\mu(t,\omega)-\mu(t,\eta)\|+\|\sigma(t,\omega)-\sigma(t,\eta)\|
\le L_N\|\omega-\eta\|_t
\]
for every $t\in[0,T]$ and all $\omega,\eta\in\Omega$ satisfying $\|\omega\|_t,\|\eta\|_t\le N$.
\label{A3}
\end{enumerate}
For $p>1$, let $Y^{(p)}$ be the (up to indistinguishability unique) solution process to the stochastic differential equation (SDE)
\begin{equation}\label{eq: SDE}
dY_t^{(p)}=\mu(t,Y^{(p)})\,dt+ \frac{1}{\sqrt{p - 1}} \sigma(t,Y^{(p)})\,dX_t,
\qquad Y_0^{(p)}=y_0,
\end{equation}
and, for $a\in\A$, let \(Z^{(p, a)}\) be the (up to indistinguishability unique) solution process to the controlled SDE
\begin{equation}\label{eq-app-controlled-small-noise-sde}
dZ_t^{(p,a)}
=\big(\mu(t,Z^{(p,a)})+ \sigma(t,Z^{(p,a)})a_t\big)\,dt
+ \frac{1}{\sqrt{p - 1}} \sigma(t,Z^{(p,a)})\,dX_t,
\qquad Z_0^{(p,a)}=y_0.
\end{equation}
At this stage, recall that \(X\) is a Brownian motion and that the claimed existence and uniqueness properties of the above equations are entailed by the Lipschitz assumptions on the coefficients, see \cite[Theorem~14.30]{jacod79} for a suitable statement.

We now obtain the following variational formula for functionals of solutions to stochastic differential equations with path-dependent coefficients:

\begin{theorem}\label{theo-VF-SDE}
Under assumptions \ref{A1}--\ref{A3}, for every universally measurable function $\varphi\colon\Omega\to[0,\infty]$ and every $p > 1$,
\begin{equation}\label{eq-main-vf-sde}
\E^\W \Big[ \varphi \big(Y^{(p)}\big)^p \Big]^{1/p}
=
\sup_{a\in\A}\E^{\mathbb W}\Big[
e^{-\frac{1}{2}\int_0^T\|a_s\|^2\,ds}\varphi \big(Z^{(p,a)}\big)
\Big].
\end{equation}
\end{theorem}

\begin{proof}
We fix \(p > 1\) and write \(Y \equiv Y^{(p)}\) and \(Z^{(a)} \equiv Z^{(p, a)}\) to simplify the notation.
	As the SDE \eqref{eq: SDE} satisfies weak existence and pathwise uniqueness, the Yamada--Watanabe theorem \cite[Corollary~5.3.23]{KaratzasShreve1991} provides the existence of a functional solution: there exists an \(\cF\)-\(\cF\)-measurable map \(F \colon \Omega \to \Omega\) that is also \(\cF^\W_t\)-\(\cF_t\)-measurable for every \(t \in [0, T]\), such that \(\W\)-a.s. \(Y = F (X)\), where we recall that the coordinate map \(X\) is a \(\W\)-Brownian motion. 
    By \cite[Lemma~I.2.17]{JS}, there also exists an \((\cF_t)_{t \in [0, T]}\)-predictable process \(G = (G_t)_{t \in [0, T]}\) on \((\Omega, \cF)\) such that \(\W (F_t = G_t \text{ for all } t \in [0, T]) = 1\).
	In particular, \(\W\)-a.s.
	\begin{align} \label{eq: strong solution}
		G_t (X) = y_0 + \int_0^t \mu (s, G (X)) \, ds + \int_0^t \frac{1}{\sqrt{p - 1}}\sigma (s, G (X)) \, d X_s, \quad t \in [0, T], 
	\end{align}
	where the stochastic integral is meant to be computed for the pair \(((\cF_t)_{t \in [0, T]}, \W)\); cf. \cite[Lemma~2.7]{jacod80}.
	Now, for every non-negative universally measurable \(\varphi \colon \Omega \to [0, \infty]\), applying Theorem~\ref{theo-VF} to the function \(\varphi \circ F\) (that is universally measurable by \cite[Proposition~7.44]{bershre}), we obtain that 
	\begin{align*}
		\E^{\W} \big[ \varphi (Y)^p \big]^{1 / p} = \E^{\W} \big[ \varphi (F (X))^p \big]^{1 / p} 
		= \sup_{a \in \A} \E^\W \Big[ e^{ - \frac{1}{2} \int_0^T \|a_s\|^2 \, ds} \, \varphi ( F ( U^{(a)})) \Big], 
	\end{align*}
	where
	\[
	U^{(a)} := X + \sqrt{p - 1} \, \int_0^\cdot a_s \, ds.
	\] 
	Thanks to \cite[Theorem~7.2]{LS_01}, we have \(\W \circ (U^{(a)})^{-1} \ll \W\), which implies that \eqref{eq: strong solution} as well as \(F = G\) hold also \(\W \circ (U^{(a)})^{-1}\)-almost surely. 
	Moreover, by \cite[Lemma~2.9]{jacod80}, \(\W\)-a.s. 
	\begin{align} \label{eq: change 1}
		\Big( \int_0^\cdot \frac{1}{\sqrt{p - 1}} \sigma (s, G (X)) \, d X_s \Big) \circ U^{(a)} 
		&= \int_0^\cdot \frac{1}{\sqrt{p - 1}}\sigma (s, G (U^{(a)})) \, d U^{(a)}_s,  
	\end{align} 
	where the latter stochastic integral is computed under $\W$ with
respect to the right-continuous natural filtration
$(\cG_t)_{t\in[0,T]}$ generated by $U^{(a)}$, given by
$\cG_t:=\bigcap_{s\in(t,T]}\sigma(U_r^{(a)},\,r\in[0,s])$ for $t<T$
and $\cG_T:=\sigma(U_r^{(a)},\,r\in[0,T])$.
    By virtue of \cite[Proposition~III.6.25]{JS}, the integral process
    \(
    \int_0^{\cdot} \, (p - 1)^{- 1/2}\sigma (s, G (U^{(a)})) \, d U^{(a)}_s
    \)
    remains \(\W\)-a.s. the same process when computed for the pair \(((\cF^\W_t)_{t \in [0, T]}, \W)\) instead of \(((\cG_t)_{t \in [0, T]}, \W)\). Hence, we further obtain that \(\W\)-a.s. 
    \begin{align} \label{eq: change 2}
    \int_0^\cdot \frac{1}{\sqrt{p - 1}} \sigma (s, G (U^{(a)})) \, d U^{(a)}_s 
		= \int_0^\cdot \sigma (s, G (U^{(a)})) \, a_s \, ds + \int_0^\cdot \frac{1}{\sqrt{p - 1}} \sigma (s, G (U^{(a)})) \, d X_s.
    \end{align} 
    Using \eqref{eq: change 1}, \eqref{eq: change 2}, and the fact that \eqref{eq: strong solution} holds \(\W \circ (U^{(a)})^{-1}\)-a.s., we conclude that \(\W\)-a.s. 
	\begin{align*}
		d G_t (U^{(a)}) &= \big( \mu (t, G (U^{(a)})) + \sigma (t, G (U^{(a)})) \, a_t \big) \, dt + \frac{1}{\sqrt{p - 1}}\sigma (t, G (U^{(a)})) \, d X_t, 
		\\
		G_0 (U^{(a)}) &= y_0.
	\end{align*} 
	As \(Z^{(a)}\) solves the same SDE on the same filtered space with the same driving Brownian motion, the pathwise uniqueness of this SDE entailed through \cite[Theorem~14.30]{jacod79} yields that \(\W\)-a.s. \(F (U^{(a)}) = G (U^{(a)}) = Z^{(a)}\). The claim of the theorem follows. 
\end{proof}

\section{Small-noise and stopping applications of the variational formula}
\label{sec-applications}
In this section, we derive three consequences of the variational formulas we established in Section~\ref{sec: VF}.
We first obtain a quantitative comparison with deterministic control,
then derive quantitative Gaussian approximations and moderate-deviation
limits, and finally establish a stopped variational formula and exact
asymptotic decay rates for events determined up to a stopping time.

Throughout this section, let $Y^{(p)}$ and $Z^{(p,a)}$ denote the solution processes of the stochastic differential equations \eqref{eq: SDE} and \eqref{eq-app-controlled-small-noise-sde},
respectively. All equalities and inequalities between random variables are understood in the $\W$-a.s.~sense.
 We write $x_+:=\max\{x,0\}$, and use the conventions
$\inf\emptyset:=\infty$, $e^{-\infty}:=0$, and $-\log 0:=\infty$.

\subsection{Comparison with deterministic control at finite \texorpdfstring{$p$}{p}}
\label{subsec-app-deterministic-reduction}
	The classical Freidlin--Wentzell theory of large deviations characterizes the exponential decay of small-noise probabilities in terms of a deterministic control problem
	\cite{FreidlinWentzell2012}. A particularly useful framework is the weak convergence approach of Dupuis and Ellis \cite{DupuisEllis1997}. Its central idea is to establish the Laplace principle by combining variational representations with a probabilistic analysis of the limiting behavior of the associated controlled processes, see, for example,
	\cite{BoueDupuis1998, BudhirajaDupuis2019, ChiariniFischer2014}.
	
	In this subsection, we apply this approach to the small-noise asymptotics of the SDE \eqref{eq: SDE}, using the variational formula from Theorem~\ref{theo-VF-SDE}. Rather than passing directly to the asymptotic limit, however, we derive explicit quantitative error estimates. 
 For this quantitative result, we assume that both coefficients \(\mu\) and \(\sigma\) are
globally Lipschitz in the path variable as described in the following conditions.
\begin{enumerate}[label=\textup{(B\arabic*)}]
\item The coefficients $\mu$ and $\sigma$ are Borel measurable and
nonanticipative in the path variable. \label{B1}
\item There is a constant $C>0$ such that, for every $t\in[0,T]$ and
every $\omega,\eta\in\Omega$,
\[
\|\mu(t,\omega)\|+\|\sigma(t,\omega)\|
\le C(1+\|\omega\|_t),
\]
\[
\|\mu(t,\omega)-\mu(t,\eta)\|
+\|\sigma(t,\omega)-\sigma(t,\eta)\|
\le C\|\omega-\eta\|_t.
\]
\label{B2}
\end{enumerate}
Conditions \ref{B1}--\ref{B2} imply \ref{A1}--\ref{A3}. The uniform Lipschitz constant is used to obtain an estimate between the stochastic and deterministic equations.

For a deterministic control
$b\in L^2([0,T];\bR^d)$, let $z^{(b)}$ be the solution of
\[
z_t^{(b)}
=
y_0+\int_0^t\mu(s,z^{(b)})\,ds
+\int_0^t\sigma(s,z^{(b)})b_s\,ds.
\]
For $a\in\A$, let $z^{(a)}$ denote the unique continuous adapted solution
of the same equation with $b$ replaced by $a$. By \ref{B1}--\ref{B2} and
the definition of $\A$, this solution is well-defined and satisfies
$z^{(a)}(\omega)=z^{(a(\cdot , \, \omega))}$ for $\W$-a.a.
$\omega\in\Omega$.

The deterministic controlled equation induces the rate function
$I\colon\Omega\to[0,\infty]$ given by
\[
I(\omega)
:=
\inf\Big\{
\frac12\int_0^T\|b_t\|^2\,dt
\colon
b\in L^2([0,T];\bR^d),\ z^{(b)}=\omega
\Big\}.
\]
Recall that $I(\, \cdot\,)$ is called a good rate function if the sublevel set $\{\omega\in\Omega\mathrel{\colon}I(\omega)\le r\}$ is compact for each $r\ge0$.

An increasing concave function $\varpi\colon[0,\infty)\to[0,\infty)$ is called a
concave modulus of continuity if $\varpi(0)=0$ and $\varpi(u)\to0$ as $u\downarrow 0$.
We say that a function $\varphi\colon\Omega\to \bR$ has concave modulus (of continuity) $\varpi$
if
\begin{align} \label{eq: concave modulus}
|\varphi(\omega)-\varphi(\eta)|
\le
\varpi(\|\omega-\eta\|_T),
\qquad
\text{for all }\omega,\eta\in\Omega.
\end{align} 
Every uniformly continuous function \(\varphi \colon \Omega \to \bR\) admits a concave modulus. In fact, the set of functions from \(\Omega \to \bR\) that has some concave modulus is precisely the class of uniformly continuous functions from \(\Omega \to \bR\). We refer to \cite[Fact~9, Lemma~12]{Joh_24} for a discussion.

The following theorem compares the stochastic and deterministic values at
every finite $p$, with an error depending on $\varphi$ only through its
modulus $\varpi$. 
\begin{theorem}\label{thm-app-finite-p-comparison}
Suppose that \ref{B1}--\ref{B2} hold. Then, $I(\, \cdot\,)$ is a good rate function.
Moreover, there is a constant $K\ge1$, depending only on $C$, $T$,
and $\|y_0\|$, such that the following hold.

\noindent
(a) For every $p\ge2$, $R\ge0$, and $a\in\A$ such that \(\W\)-a.s. 
$\int_0^T\|a_t\|^2\,dt\le R$, 
\begin{equation}\label{eq-app-controlled-deterministic-error}
\E^\W\Big[
\|Z^{(p,a)}-z^{(a)}\|_T^2
\Big]^{1/2}
\le
\frac{K e^{K\sqrt R}}{\sqrt{p-1}}.
\end{equation}
(b) If $p\ge2$ and $\varphi\colon\Omega\to[0,\infty)$ has concave modulus
$\varpi$, then $e^{-I}\varphi$ attains a finite maximum, and
\begin{equation}\label{eq-app-modulus-comparison}
\Big|
\E^\W\Big[\varphi(Y^{(p)})^p\Big]^{1/p}
-
\max_{\omega\in\Omega}e^{-I(\omega)}\varphi(\omega)
\Big|
\leq
\varpi\Big(\frac{K}{\sqrt{p-1}}\Big).
\end{equation}
\end{theorem}

\begin{proof}
Throughout the proof, $c\ge 0$ denotes a finite constant that may change from
line to line and depends only on $C$, $T$, and $\|y_0\|$.

\smallskip
{\em Step 1.} We consider the deterministic controlled equation. For $R \geq 0$, we set
\[
B_R
:=
\Big\{
b\in L^2([0,T];\bR^d)
\colon
\int_0^T\|b_t\|^2\,dt\le R
\Big\},
\qquad
K_R:=\{z^{(b)}\ \colon b\in B_R\}.
\]
The space $B_R$ endowed with the weak topology of
$L^2([0,T];\bR^d)$ is compact and metrizable. We show that the solution map $b\mapsto z^{(b)}$ from $B_R$ to $\Omega$ is continuous. Indeed, let
$b^n\rightharpoonup b$ in $B_R$ and set
$r_n(t):=\int_0^t\sigma(s,z^{(b)})(b_s^n-b_s)\,ds$. The growth bound in \ref{B2} implies
that 
$\sigma(\, \cdot \,,z^{(b)})\in L^2([0,T];\bR^{d\times d})$, which shows that 
$r_n(t)\to 0$ for all $t\in[0,T]$. Since $b^n,b\in B_R$, we have $\|b^n-b\|_{L^2}\le 2\sqrt R$, and
H\"older's inequality implies
$\|r_n(t)-r_n(s)\|\le2\sqrt R
(\int_s^t\|\sigma(u,z^{(b)})\|^2\,du)^{1/2}$ for all $0 \leq s \leq t \leq T$.
The absolute continuity of the Lebesgue integral
implies that $(r_n)_{n\in\mathbb{N}}$ is equicontinuous,
from which we deduce that $\|r_n\|_T\to 0$. To compare the two solutions, we add and subtract
$\int_0^t \sigma(s,z^{(b)})b_s^n\,ds$ and obtain
\[
\begin{aligned}
z_t^{(b^n)}-z_t^{(b)}
&=r_n(t)
+\int_0^t\bigl(\mu(s,z^{(b^n)})-\mu(s,z^{(b)})\bigr)\,ds
+\int_0^t\bigl(\sigma(s,z^{(b^n)})-\sigma(s,z^{(b)})\bigr)b_s^n\,ds.
\end{aligned}
\]
Taking the supremum up to time $t\in[0,T]$ and applying \ref{B2}, we obtain
\[
\|z^{(b^n)}-z^{(b)}\|_t
\le
\|r_n\|_T
+c\int_0^t\|z^{(b^n)}-z^{(b)}\|_s(1+\|b_s^n\|)\,ds.
\]
By H\"older's inequality, we have
$\int_0^T(1+\|b_s^n\|)\,ds\le T+\sqrt{TR}$. Hence, it follows from Gronwall's lemma that  $\|z^{(b^n)}-z^{(b)}\|_T
\le\|r_n\|_T\exp\bigl(c(T+\sqrt{TR})\bigr)\to 0$.
Thus, $b\mapsto z^{(b)}$ is continuous on $B_R$, and therefore its image $K_R$ is compact.

\smallskip
{\em Step 2.} We show that the infimum in the rate function  $I(\, \cdot \,)$ is attained. To do so, fix $\omega\in \Omega$ with
$I(\omega)<\infty$, and choose $b^n\in L^2([0,T];\bR^d)$ such that
$z^{(b^n)}=\omega$ and
\[
I(\omega)
\le
\frac12\int_0^T\|b_t^n\|^2\,dt
\le
I(\omega)+\frac1n.
\]
The sequence $(b^n)_{n\in\mathbb N}$ is contained in $B_R$ with \(R = 2 (I (\omega) + 1)\).
After passing to a subsequence, $b^n\rightharpoonup b$ for some
$b\in B_R$, so Step~1 implies $z^{(b)}=\omega$. Since the $L^2$-norm is weakly lower semicontinuous, it follows that
\[
\frac12\int_0^T\|b_t\|^2\,dt
\le
\liminf_{n\to\infty}
\frac12\int_0^T\|b_t^n\|^2\,dt
=
I(\omega).
\]
The reverse inequality follows from the definition of $I(\, \cdot \,)$, so $b$ attains
the infimum. Consequently, for every $r\ge 0$, the sublevel set $\{\omega\in\Omega\mathrel{\colon}I(\omega)\le r\}=K_{2r}$ is compact by the previous step. In particular, the rate function $I(\, \cdot \,)$ is lower semicontinuous.

\smallskip
{\em Step 3.} Set
$M:=\sup_{\omega\in\Omega}e^{-I(\omega)}\varphi(\omega)$.
We first show that $M$ is finite and the supremum is attained. 
First, using \ref{B2}, we obtain that 
\[
\| z^{(b)}\|_t\leq  \| y_0\| + C \int_0^t \big( 1 + \| z^{(b)} \|_s \big) \, ds +  C \int_0^t \big( 1 + \| z^{(b)} \|_s \big) \| b_s\| \, ds, 
\] 
and hence, Gronwall's lemma yields
\begin{align} \label{eq: strong moment}
1 +	\| z^{(b)} \|_T \leq ce^{c \int_0^T \| b_s \| \, ds} \leq c e^{c\sqrt{T} \| b \|_{L^2}}.
\end{align} 
 If $I(\omega)<\infty$, Step~2
provides a minimizing control $b\in L^2([0,T];\bR^d)$ with
$\|b\|_{L^2}=\sqrt{2I(\omega)}$, and hence
$1+\|\omega\|_T\le c\exp(c\sqrt{I(\omega)})$. Moreover, concavity and
$\varpi(0)=0$ imply
$\varpi(u)\le\varpi(1)(1+u)$ for all $u\ge0$. The modulus condition
therefore implies
\[
\varphi(\omega)
\le
\varphi(0)+\varpi(\|\omega\|_T)
\le
c_\varphi(1+\|\omega\|_T),
\qquad
c_\varphi:=\varphi(0)+\varpi(1).
\]
Consequently, we obtain
\[
e^{-I(\omega)}\varphi(\omega)
\le
c\,c_\varphi
\exp\Big(-I(\omega)+c\sqrt{I(\omega)}\Big)
\to 0
\qquad\text{as }I(\omega)\to\infty.
\]
Since $I(\, \cdot \,)$ is lower semicontinuous, $e^{-I}\varphi$ is upper semicontinuous, and bounded above on every compact
sublevel set of $I(\, \cdot \,)$. Together with the preceding estimate we obtain
$M<\infty$. If $M=0$, attainment is immediate by nonnegativity. If
$M>0$, every maximizing sequence is eventually contained in a compact
sublevel set of $I(\, \cdot \,)$, and upper semicontinuity yields a maximizer. Thus,
\[
M
=
\max_{\omega\in\Omega}e^{-I(\omega)}\varphi(\omega).
\]

We next identify the value $M$ with the deterministic and random control problem.
For every $b\in L^2([0,T];\bR^d)$, by definition of $I(\, \cdot \,)$ we get
$I(z^{(b)})\le\frac12\int_0^T\|b_t\|^2\,dt$, and therefore
\begin{equation}\label{eq:detcontrol}
e^{-\frac12\int_0^T\|b_t\|^2\,dt}\varphi(z^{(b)})
\le
e^{-I(z^{(b)})}\varphi(z^{(b)})
\le
M.
\end{equation}
Hence, the supremum over such $b$ is at most $M$. Let $\omega^\ast$
attain $M$. If $M>0$, then $I(\omega^\ast)<\infty$, and Step~2 guarantees
$b^\ast\in L^2([0,T];\bR^d)$ with
$z^{(b^\ast)}=\omega^\ast$ and
$\frac12\int_0^T\|b_t^\ast\|^2\,dt=I(\omega^\ast)$.
Both inequalities in \eqref{eq:detcontrol} are then equalities for
$b^\ast$. If $M=0$, it follows from \eqref{eq:detcontrol} that $b^\ast=0$ attains
$M$. Thus, the supremum equals $M$ and is attained.

We finally consider random controls. For $a\in\A$ and $\W$-a.a.~$\omega\in\Omega$, the control $a(\, \cdot \,,\omega)$ belongs
to $L^2([0,T];\bR^d)$ and $z^{(a)}(\omega)$ solves the deterministic
controlled equation with control $a(\, \cdot\,,\omega)$. Hence, $z^{(a)}(\omega)=z^{(a(\cdot, \,\omega))}$, and applying \eqref{eq:detcontrol} to the realized control $a(\, \cdot\, ,\omega)$ shows that
$\exp(-\frac12\int_0^T\|a_t(\omega)\|^2\,dt)\varphi(z^{(a)}(\omega))\le M$.
Hence, taking expectations and then the supremum over $a\in\A$ yields
\begin{equation}\label{eq-app-deterministic-value}
\max_{a\in\A}
\E^\W\Big[
e^{-\frac12\int_0^T\|a_t\|^2\,dt}\varphi(z^{(a)})
\Big]
=
\max_{b\in L^2([0,T];\bR^d)}
e^{-\frac12\int_0^T\|b_t\|^2\,dt}\varphi(z^{(b)})
=
\max_{\omega\in\Omega}e^{-I(\omega)}\varphi(\omega),
\end{equation}
because the supremum over random controls dominates the supremum over deterministic controls.

\smallskip
{\em Step 4.} 
We prove
\eqref{eq-app-controlled-deterministic-error}. 
Fix $R\ge0$ and
$a\in\A$ satisfying \(\W\)-a.s.
$\int_0^T\|a_s\|^2\,ds\le R$. 
From \eqref{eq: strong moment} we obtain that \(\W\)-a.s. 
\begin{align} \label{eq: moment}
\| z^{(a)} \|_T \leq c e^{c\sqrt{TR}}.
\end{align} 
Similarly, using \ref{B2}, the Cauchy--Schwarz inequality, and the Burkholder--Davis--Gundy inequality, we find that 
\begin{align*}
	\E^\W \Big[ \| Z^{(p, a)} \|^2_{t \wedge \tau_n} \Big] &\leq 4\| y_0 \|^2 + c ( 1 + R )\, \E^\W \Big[ \int_0^{t \wedge \tau_n} \big( 1 + \| Z^{(p, a)} \|^2_s \big) \, ds \Big], \quad t \in [0, T], 
\end{align*}
where \(\tau_n := \inf \{ t \in [0, T] \colon \| Z^{(p, a)}_t \| \geq n \}\). In view of this inequality, Gronwall's lemma yields that 
\[
\E^\W \Big[ \| Z^{(p, a)} \|^2_T \Big] \leq \liminf_{n \to \infty} \E^\W \Big[ \| Z^{(p, a)} \|^2_{T \wedge \tau_n} \Big] \leq (1 + 4 \| y_0 \|^2) e^{ c (1 + R) T} < \infty. 
\]
In particular, we have
\begin{align} \label{eq: finite}
\E^\W \Big[ \| Z^{(p, a)} - z^{(a)} \|_T^2 \Big] < \infty.
\end{align} 
In the following we use a time-change argument to reduce the rate \(e^{R}\) in the above Gronwall argument by \(e^{\sqrt{R}}\).
For \(t \in [0, T]\) and \(u \in \bR_+\), we set 
\[
A_t := \int_0^t (1 + \|a_s\| \big) \, ds, \quad L_u:= \inf \{s \in [0, T] \colon A_s > u \} \wedge T. 
\] 
It is well-known that \((L_u)_{u \geq 0}\) is an increasing family of stopping times. 
Furthermore, \(\W\)-a.s. \(L_u = T\) for all \(u \geq T + \sqrt{T R}\), and \(A_{L_t} - A_{L_s} \leq t - s\) for all \(s < t\). Indeed, to explain the second claim, if \(L_s \geq T\), the inequality is trivial, and if \(L_s < T\), then \(A_{L_s} = s\) and, because \(A_{L_t} \leq t\), the inequality holds again. From this fact, it follows that 
\[
\int_0^\infty g (s) \, d A_{L_s} \leq \int_0^\infty g (s) \, ds
\] 
for all non-negative measurable functions \(g \colon \bR_+ \to \bR_+\).
Using It\^o's formula, once again \ref{B2}, the change of variable rule \cite[Proposition~4.10, p.~9]{RY}, and the bound \eqref{eq: moment}, we obtain, \(\W\)-a.s. for all \(t \in \bR_+\),  
\begin{align*}
	\| &Z^{(p,a)}_{L_t} -z^{(a)}_{L_t}\|^2 
	\\&= 2 \int_0^{L_t} \langle Z^{(p, a)}_s - z^{(a)}_s, \mu (s, Z^{(p, a)}) - \mu (s, z^{(a)}) + ( \sigma (s, Z^{(p, a)}) - \sigma (s, z^{(a)}) ) a_s \rangle \, ds 
	\\&\qquad + \frac{2}{\sqrt{p - 1}} \int_0^{L_t} \langle Z^{(p, a)}_s - z^{(a)}_s, \sigma (s, Z^{(p, a)}) \, d X_s \rangle 
 + \frac{1}{p - 1} \int_0^{L_t}\| \sigma (s, Z^{(p, a)})\|^2 \, ds
	\\&\leq c  \int_0^{L_t} \|Z^{(p, a)} - z^{(a)} \|^2_s \, d A_s + \frac{2}{\sqrt{p - 1}} \int_0^{L_t} \langle Z^{(p, a)}_s - z^{(a)}_s, \sigma (s, Z^{(p, a)}) \, d X_s \rangle 
	\\&\qquad + \frac{1}{p - 1} \int_0^{L_t}\| \sigma (s, Z^{(p, a)})\|^2 \, ds
	\\&\leq c  \int_0^{t} \|Z^{(p, a)} - z^{(a)} \|^2_{L_s} \, d s + \frac{2}{\sqrt{p - 1}} \int_0^{L_t} \langle Z^{(p, a)}_s - z^{(a)}_s, \sigma (s, Z^{(p, a)}) \, d X_s \rangle 
	\\&\qquad + \frac{2}{p - 1} \int_0^{L_t}\| \sigma (s, z^{(a)})\|^2 \, ds
		\\&\leq c  \int_0^{t} \|Z^{(p, a)} - z^{(a)} \|^2_{L_s} \, d s + \frac{2}{\sqrt{p - 1}} \int_0^{L_t} \langle Z^{(p, a)}_s - z^{(a)}_s, \sigma (s, Z^{(p, a)}) \, d X_s \rangle 
+ \frac{c}{p - 1} e^{c \sqrt{R}}, 
\end{align*}
where we recall that the constant \(c\) may have changed from line to line.
Using the Burkholder--Davis--Gundy inequality, Young's inequality, and the bound \eqref{eq: moment}, we obtain that
\begin{align*}
	\E^\W \Big[ \sup_{s \in [0, t]} \Big| & \frac{2}{\sqrt{p-1}} \int_0^{L_s} \langle Z^{(p, a)}_u - z^{(a)}_u, \sigma (u, Z^{(p, a)}) \, d X_u \rangle \Big| \Big] 
	\\&\leq  \frac{c}{\sqrt{p - 1}} \E^\W\Big[ \Big( \int_0^{L_t} \| \sigma^* (s, Z^{(p, a)}) ( Z^{(p, a)}_s - z^{(a)}_s ) \|^2 \, ds \Big)^{1/2} \Big]
		\\&\leq  \frac{c}{\sqrt{p - 1}} \E^\W\Big[ \| Z^{(p, a)} - z^{(a)} \|_{L_t} \, \Big( \int_0^{L_t} \| \sigma (s, Z^{(p, a)}) \|^2 \, ds \Big)^{1/2} \Big]
	\\&\leq \frac{1}{2} \, \E^\W \Big[ \| Z^{(p, a)} - z^{(a)} \|^2_{L_t} \Big] + \frac{c^2}{2(p - 1)} \E^\W \Big[ \int_0^{L_t} \| \sigma (s, Z^{(p, a)} ) \|^2 \, ds \Big]
	\\&\leq \frac{1}{2} \, \E^\W \Big[ \| Z^{(p, a)} - z^{(a)} \|^2_{L_t} \Big] + c\, \E^\W \Big[ \int_0^{t} \| Z^{(p, a)} - z^{(a)} \|^2_{L_s} \, ds \Big] + \frac{c}{p - 1} e^{c \sqrt{R}}.
	\end{align*}
	In summary, we obtain, for all \(t \in \bR_+\),	
\begin{equation} \label{eq: last inequ} \begin{split}
	\E^{\W} \Big[  \|Z^{(p,a)}-z^{(a)}\|_{L_t}^2 \Big] \leq c \int_0^t \E^\W \Big[ \|Z^{(p,a)}-z^{(a)}\|_{L_s}^2 \Big]  \, ds &+ \frac{c}{p - 1} e^{c \sqrt{R}} \\&+ \frac{1}{2} \, \E^\W \Big[ \| Z^{(p, a)} - z^{(a)} \|^2_{L_t} \Big].
\end{split} \end{equation} 
Recalling \eqref{eq: finite}, we bring the last term in \eqref{eq: last inequ} on the left-hand side and use Gronwall's inequality to find, for all \(t \in \bR_+\), 
\begin{align*}
	\E^{\W} \Big[ \|Z^{(p,a)}-z^{(a)}\|_{L_t}^2 \Big] \leq \frac{c}{p - 1} e^{c \sqrt{R}} e^{c t}.
\end{align*}
Using this inequality for \(t = T + \sqrt{T R}\) yields \eqref{eq-app-controlled-deterministic-error}.

\smallskip
{\em Step 5.} We extend the estimate from Step~4 to arbitrary controls by truncation and use that $\varphi$ has modulus $\varpi$. Fix
$a\in\A$. For $n\in\mathbb{N}_0$, we define
\[
\begin{aligned}
E_n
&:=
\Big\{
\omega\in\Omega
\mathrel{\colon}
n\le\int_0^T\|a_s(\omega)\|^2\,ds<n+1
\Big\},\\
\rho_n
&:=
\inf\Big\{
t\in[0,T]
\colon
\int_0^t\|a_s\|^2\,ds\ge n+1
\Big\}\wedge T,
\qquad
a_t^{(n)}:=a_t\1_{\{t\le\rho_n\}} \in \mathcal{A}.
\end{aligned}
\]
Thanks to the Lipschitz assumptions on \(\mu\) and \(\sigma\), we have \(\W\)-a.s. 
\[ Z^{(p, a)}_t = Z^{(p, a^{(n)})}_t \quad \text{and}\quad z^{(a)}_t = z^{(a^{(n)})}_t\] for all \(t \leq \rho_n\). Moreover, on \(E_n\) we have \(T \leq \rho_n\). Using these facts, Fubini's theorem, that \(\W\)-a.s. \(\bigcup_{n = 0}^\infty E_n = \{ \int_0^T \| a_s \|^2 \, ds < \infty \} = \Omega\), and that \(\W\)-a.s. \(\int_0^T \| a^{(n)}_s\|^2 \,ds \leq n + 1\), we obtain that 
\begin{equation} \label{eq: summation} \begin{split}
	\E^\W \Big[ e^{- \frac{1}{2} \int_0^T \| a_s \|^2 \, ds } \| Z^{(p, a)} - z^{(a)} \|_T \Big] &= \sum_{n = 0}^\infty 	\E^\W \Big[ e^{- \frac{1}{2} \int_0^T \| a_s \|^2 \, ds } \| Z^{(p, a)} - z^{(a)} \|_T \1_{E_n} \Big]
	\\&\leq \sum_{n = 0}^\infty e^{- n /2}	\E^\W \Big[ \| Z^{(p, a^{(n)})} - z^{(a^{(n)})} \|_T \Big] 
	\\&\leq \sum_{n = 0}^\infty \frac{c}{\sqrt{p - 1}} e^{c \sqrt{ n + 1 } - n / 2} \leq \frac{c}{\sqrt{p - 1}}, 
\end{split}
\end{equation} 
where \(c\) may have changed in the last inequality.
Using Jensen's inequality under the probability measure
$q^{-1}\exp(-\frac12\int_0^T\|a_s\|^2\,ds)\,d\W$ with \(q:=\E^\W[\exp(-\frac12\int_0^T\|a_s\|^2\,ds)]\), and that $q \varpi(x/q)\le \varpi(x)$ for all $x\ge 0$, we finally find that 
\begin{align*}
\E^\W\Big[
e^{-\frac12\int_0^T\|a_s\|^2\,ds}
\big|\varphi(Z^{(p,a)})-\varphi(z^{(a)})\big|
\Big]
&\le
q \varpi\Big(
\frac1q
\E^\W\Big[
e^{-\frac12\int_0^T\|a_s\|^2\,ds}
\|Z^{(p,a)}-z^{(a)}\|_T
\Big]
\Big)
\\&\leq 
\varpi\Big(
\E^\W\Big[
e^{-\frac12\int_0^T\|a_s\|^2\,ds}
\|Z^{(p,a)}-z^{(a)}\|_T
\Big]
\Big)
\\&\le
\varpi\Big(\frac{c}{\sqrt{p-1}}\Big).
\end{align*}
Finally, Theorem~\ref{theo-VF-SDE} and 
\eqref{eq-app-deterministic-value} therefore yield
\begin{align*}
\Big|
\E^\W\Big[\varphi(Y^{(p)})^p\Big]^{1/p}
-
\max_{\omega\in\Omega}e^{-I(\omega)}\varphi(\omega)
\Big|
&\leq
\sup_{a\in\A}
\E^\W\Big[
e^{-\frac12\int_0^T\|a_s\|^2\,ds}
\big|\varphi(Z^{(p,a)})-\varphi(z^{(a)})\big|
\Big]
\\&\leq
\varpi\Big(\frac{c}{\sqrt{p-1}}\Big),
\end{align*}
which shows \eqref{eq-app-modulus-comparison}. The proof is complete.
\end{proof}

\begin{remark}\label{rem-app-finite-p-rates}
If $\varphi$ is $\alpha$-H\"older continuous with constant
$L_\varphi$, where $\alpha\in(0,1]$, then
$\varpi(u)=L_\varphi u^\alpha$ is a concave modulus of continuity, and
the right-hand side of \eqref{eq-app-modulus-comparison} equals
$K^\alpha L_\varphi/(p-1)^{\alpha/2}$. In particular, the error is of
order $1/(p-1)^{1/2}$ for Lipschitz functionals. 

While this excludes indicator functions, the theorem also applies
to families of functionals depending on $p$. More precisely, if each
$\varphi_p$ has a concave modulus $\varpi_p$, then the corresponding
stochastic and deterministic values differ by a quantity tending to
zero whenever $\varpi_p(K/\sqrt{p-1})\to 0$.
\end{remark}

\subsubsection*{Tracking of trajectories}
We next focus on the corresponding results for sets. For $A\subset\Omega$, let
$I(A):=\inf_{\omega\in A}I(\omega)$. Moreover, for
$\delta>0$, we define
\[
A^\delta
:=
\{
\omega\in\Omega
\mathrel{\colon}
d_T(\omega,A)<\delta
\},
\qquad
A_{-\delta}
:=
\{
\omega\in A
\mathrel{\colon}
d_T(\omega,A^c)\ge\delta
\},
\]
where $d_T(\omega,A):=\inf_{\eta\in A}\|\omega-\eta\|_T$ defines the distance function with the convention
$d_T(\omega,\emptyset):=\infty$. 
The following two-sided finite-$p$ bounds compare the probability of a set with the deterministic costs of its inner and outer neighborhoods.

\begin{corollary}\label{cor-app-event-neighborhoods}
Suppose that \ref{B1}--\ref{B2} hold, and let $K$ be the constant from
Theorem~\ref{thm-app-finite-p-comparison}. Then, for every universally
measurable set $B\subset\Omega$, every $\delta>0$, and every $p\ge2$,
\begin{equation}\label{eq-app-event-neighborhoods}
\Big(
e^{-I(B_{-\delta})}
-\frac{K}{\delta\sqrt{p-1}}
\Big)_+
\le
\W(Y^{(p)}\in B)^{1/p}
\le
e^{-I(B^\delta)}
+\frac{K}{\delta\sqrt{p-1}}.
\end{equation}
\end{corollary}

\begin{proof}
The assertion is immediate for $B=\emptyset$. Otherwise, set
\[
\varphi_\delta^-(\omega)
:=
\min\Big\{
1,\frac{d_T(\omega,B^c)}{\delta}
\Big\},
\qquad
\varphi_\delta^+(\omega)
:=
\Big(
1-\frac{d_T(\omega,B)}{\delta}
\Big)_+.
\]
Both functions are $[0,1]$-valued and $1/\delta$-Lipschitz, and
$\varphi_\delta^-\le\1_B\le\varphi_\delta^+$. Moreover,
$\varphi_\delta^-=1$ on $B_{-\delta}$, whereas
$\varphi_\delta^+$ vanishes outside $B^\delta$. 
Applying Theorem~\ref{thm-app-finite-p-comparison} with the concave
modulus $\varpi(u)=u/\delta$, $u\ge0$, yields
\[
\begin{aligned}
\W(Y^{(p)}\in B)^{1/p}
&\ge
\E^\W\Big[
\varphi_\delta^-(Y^{(p)})^p
\Big]^{1/p}
\ge
\max_{\omega\in\Omega}
e^{-I(\omega)}\varphi_\delta^-(\omega)
-\frac{K}{\delta\sqrt{p-1}}
\ge
e^{-I(B_{-\delta})}
-\frac{K}{\delta\sqrt{p-1}},\\
\W(Y^{(p)}\in B)^{1/p}
&\le
\E^\W\Big[
\varphi_\delta^+(Y^{(p)})^p
\Big]^{1/p}
\le
\max_{\omega\in\Omega}
e^{-I(\omega)}\varphi_\delta^+(\omega)
+\frac{K}{\delta\sqrt{p-1}}
\le
e^{-I(B^\delta)}
+\frac{K}{\delta\sqrt{p-1}}.
\end{aligned}
\]
Since the left-hand side is non-negative, the first estimate may be replaced
by its positive part, which proves \eqref{eq-app-event-neighborhoods}.
\end{proof}

\subsection{Gaussian approximation and moderate deviations}
\label{subsec-app-moderate-deviations}
In Subsection~\ref{subsec-app-deterministic-reduction} we studied a quantitative \(L^p\)-version of the Freidlin--Wentzell limit. In this section we proceed with a quantitative \(L^p\)-analysis of moderate deviations. 
More precisely, we center 
$Y^{(p)}$ at $z^{(0)}$ and investigate deviations on the scale $h/\sqrt{p-1}$, with $1\le h\le\sqrt{p-1}$. Then, we compare the rescaled
control problem at moment order $1+h^2$ with a linear control problem
whose stochastic part is Gaussian. The parameter $h$ connects the
Gaussian scale $h=1$ with the nonlinear scale $h=\sqrt{p-1}$.

Throughout this section, we will work under the assumptions \ref{B1}--\ref{B2}.
Fix $p\ge2$ and $1\le h\le\sqrt{p-1}$, and set
$\bar p:=1+h^2$ and $\bar\sigma:=h\sigma/\sqrt{p-1}$. Let
$\bar Y^{(\bar p)}$ and $\bar Z^{(\bar p,a)}$ denote the uncontrolled
and controlled processes in Theorem~\ref{theo-VF-SDE} with coefficients
$(\mu,\bar\sigma)$. Since
\[
\frac{\bar\sigma}{\sqrt{\bar p-1}}
=
\frac{\sigma}{\sqrt{p-1}},
\qquad
\bar\sigma a
=
\frac h{\sqrt{p-1}}\sigma a,
\]
the Lipschitz condition \ref{B2} entails that
$\bar Y^{(\bar p)}=Y^{(p)}$ and
$\bar Z^{(\bar p,a)}=Z^{(p,\frac h{\sqrt{p-1}}a)}$.
Moreover, we set
\[
\Xi^{(p,h,a)}
:=
\frac{\sqrt{p-1}}h
\Big(
\bar Z^{(\bar p,a)}-z^{(0)}
\Big),
\qquad
\Xi^{(p,h)}
:=
\Xi^{(p,h,0)}=
\frac{\sqrt{p-1}}h
\Big(
Y^{(p)}-z^{(0)}
\Big).
\]
For every non-negative universally measurable functional $\varphi \colon \Omega \to \bR_+$,
it follows from Theorem~\ref{theo-VF-SDE}, applied to
$\omega\mapsto\varphi(\sqrt{p-1} \, (\omega-z^{(0)})/h)$, that
\begin{equation}\label{eq-app-rescaled-variational-formula}
\E^\W\Big[
\varphi\bigl(\Xi^{(p,h)}\bigr)^{\bar p}
\Big]^{1/\bar p}
=
\sup_{a\in\A}
\E^\W\Big[
e^{-\frac12\int_0^T\|a_t\|^2\,dt}
\varphi\big (\Xi^{(p,h,a)}\big)
\Big].
\end{equation}
Since $z_t^{(0)}=y_0+\int_0^t\mu(s,z^{(0)})\,ds$, the rescaled process $\Xi^{(p,h,a)}$ has dynamics

\begin{equation} \label{eq: Xi} \begin{split}
\Xi^{(p,h,a)}_t
&=
\frac{\sqrt{p-1}}h
\int_0^t
\Big[
\mu\Big(
s,z^{(0)}+\frac h{\sqrt{p-1}}\Xi^{(p,h,a)}
\Big)
-\mu(s,z^{(0)})
\Big]ds\\
&\qquad+
\int_0^t
\sigma\Big(
s,z^{(0)}+\frac h{\sqrt{p-1}}\Xi^{(p,h,a)}
\Big)a_s\,ds
\\&\qquad+\frac1h
\int_0^t
\sigma\Big(
s,z^{(0)}+\frac h{\sqrt{p-1}} \Xi^{(p,h,a)}
\Big) \, dX_s.
\end{split} 
\end{equation}
The dynamics \eqref{eq: Xi} distinguish three regimes. The Brownian term is
multiplied by $1/h$, while condition \ref{B3} below controls the
linearization error in terms of $h/\sqrt{p-1}$. For $h=1$, the Brownian
term remains as $p\to\infty$, while the linearization error vanishes,
leading to the Gaussian approximation below. If $h\to\infty$ and
$h/\sqrt{p-1}\to0$, both the Brownian term and the linearization error
vanish, and the deterministic linear control problem determines the
moderate-deviation asymptotics. Finally, if $h=\sqrt{p-1}$, then
$\bar p=p$, $\bar Z^{(\bar p,a)}=Z^{(p,a)}$, and
$\Xi^{(p,h,a)}=Z^{(p,a)}-z^{(0)}$. Thus, we recover the nonlinear
controlled equation from
Subsection~\ref{subsec-app-deterministic-reduction}, and the
linearization error need not vanish.
Theorem~\ref{thm-app-moderate-comparison} quantifies the approximation
errors in these three regimes.
Gaussian approximations of small-noise diffusions in relative entropy
are studied in \cite{SanzAlonsoStuart2017}. Moderate-deviation
principles based on linearization are classical
\cite[Chapters~9--10]{BudhirajaDupuis2019}, and pathwise results for
diffusion and delay equations include
\cite{JacquierSpiliopoulos2020,SuoTaoZhang2018}.  

After increasing the constant $C$ in \ref{B2} if necessary, we
assume that the drift admits the following linear expansion at
$z^{(0)}$.
\begin{enumerate}[label=\textup{(B\arabic*)},start=3]
\item There is a Borel nonanticipative map
$A\colon[0,T]\times\Omega\to\bR^d$, linear in its second argument,
such that, writing $A_t\eta=A(t,\eta)$,
\[
\|A_t\eta\|
\le C\|\eta\|_t,
\qquad
\|\mu(t,z^{(0)}+\eta)-\mu(t,z^{(0)})-A_t\eta\|
\le C\|\eta\|_t^2
\]
for all $t\in[0,T]$ and $\eta\in\Omega$.
\label{B3}
\end{enumerate}
In \eqref{eq: Xi}, the coefficients are evaluated at
$z^{(0)}+h\Xi^{(p,h,a)}/\sqrt{p-1}$. Hence,
\ref{B3} bounds the difference between the rescaled drift and
$A_t\Xi^{(p,h,a)}$ by
$Ch\|\Xi^{(p,h,a)}\|_t^2/\sqrt{p-1}$, while \ref{B2} bounds
the difference between the diffusion coefficient and
$\sigma(t,z^{(0)})$ by
$Ch\|\Xi^{(p,h,a)}\|_t/\sqrt{p-1}$. Replacing the drift by its linear
part and freezing the diffusion coefficient at $z^{(0)}$ leads to
the approximation $v^{(a)}+G/h$ of $\Xi^{(p,h,a)}$, where $G$ and $v^{(a)}$ solve
\[
\begin{aligned}
G_t
&=
\int_0^tA_sG\,ds
+\int_0^t\sigma(s,z^{(0)})\,dX_s,\\
v_t^{(a)}
&=
\int_0^tA_sv^{(a)}\,ds
+\int_0^t\sigma(s,z^{(0)})a_s\,ds,\qquad a\in\A.
\end{aligned}
\]
Since the equation for $G$ is linear with deterministic coefficients,
$G$ is a centered Gaussian process.
Our goal is to apply Theorem~\ref{theo-VF-SDE} at moment order
$\bar p=1+h^2$, with initial value zero and linear coefficients
\[
(t,\eta)\mapsto A_t\eta,
\qquad
(t,\eta)\mapsto\sigma(t,z^{(0)}).
\]
Since $\sqrt{\bar p-1}=h$, the Brownian coefficient in the corresponding
uncontrolled equation is $\sigma(t,z^{(0)})/h$, while the control adds
the drift $\sigma(t,z^{(0)})a_t$. Since the drift is linear, its uncontrolled and controlled
solutions are therefore $G/h$ and
$\Gamma^{(h,a)}=v^{(a)}+G/h$, respectively. 

As $h\to\infty$, the Gaussian term $G/h$ vanishes and the linear
control problem becomes deterministic. Its rate function is given by 
\[
I_{\mathrm{lin}}(\eta)
:=
\inf \Big\{
\frac12\int_0^T\|b_t\|^2\,dt
\colon
b\in L^2([0,T];\bR^d),\ v^{(b)}=\eta
\Big\}, 
\]
and we set $I_{\mathrm{lin}}(B):=\inf_{\eta\in B}I_{\mathrm{lin}}(\eta)$
for $B\subseteq\Omega$.
Theorem~\ref{thm-app-finite-p-comparison}, applied to the linear coefficients, shows that
$I_{\mathrm{lin}}(\, \cdot\, )$ is a good rate function and that, whenever
$\varphi$ is non-negative and has a concave modulus (recall \eqref{eq: concave modulus}),
$e^{-I_{\mathrm{lin}}}\varphi$ attains a finite maximum. Moreover,
Theorem~\ref{theo-VF-SDE} ensures
\begin{equation}\label{eq-app-linear-values}
\E^\W\Big[
\varphi\Big(\frac Gh\Big)^{\bar p}
\Big]^{1/\bar p}
=
\sup_{a\in\A}
\E^\W\Big[
e^{-\frac12\int_0^T\|a_t\|^2\,dt}
\varphi\bigl(\Gamma^{(h,a)}\bigr)
\Big].
\end{equation}
The right-hand sides of
\eqref{eq-app-rescaled-variational-formula} and
\eqref{eq-app-linear-values} differ only through the controlled paths
$\Xi^{(p,h,a)}$ and $\Gamma^{(h,a)}$. The following theorem compares
these paths uniformly over all controls \(a \in \mathcal{A}\) and then identifies the
deterministic linear value.
\begin{theorem}\label{thm-app-moderate-comparison}
Suppose that \ref{B1}--\ref{B3} hold. There is a constant $K\ge1$,
depending only on $C$, $T$, and $\|y_0\|$, such that, for every
$p\ge2$ and $1\le h\le\sqrt{p-1}$,
\begin{equation}\label{eq-app-weighted-gaussian-comparison}
\begin{aligned}
\E^\W\Big[
\Big\|
\Xi^{(p,h)}-\frac Gh
\Big\|_T^{1+h^2}
\Big]^{1/(1+h^2)}
&=
\sup_{a\in\A}
\E^\W\Big[
e^{
-\frac12\int_0^T\|a_t\|^2\,dt
}
\Big\|
\Xi^{(p,h,a)}-\Gamma^{(h,a)}
\Big\|_T
\Big]\\
&\le
\frac{Kh}{\sqrt{p-1}}.
\end{aligned}
\end{equation}
If $\varphi\colon\Omega\to[0,\infty)$ has concave modulus $\varpi$,
then
\begin{equation}\label{eq-app-moderate-comparison}
\Big|
\E^\W\Big[
\varphi\big(
\Xi^{(p, h)}
\big)^{1+h^2}
\Big]^{1/(1+h^2)}
-
\max_{\eta\in\Omega}
e^{-I_{\mathrm{lin}}(\eta)}\varphi(\eta)
\Big|
\le
\varpi\Big(
K\Big(
\frac1h+\frac h{\sqrt{p-1}}
\Big)
\Big).
\end{equation}
\end{theorem}

\begin{proof}
Throughout the proof, $c$ denotes a finite constant that may change
from line to line and depends only on $C$, $T$, and $\|y_0\|$.

We first prove
\eqref{eq-app-weighted-gaussian-comparison}. Set
$\varepsilon:=h/\sqrt{p-1}$ and define
\[
\mu_\varepsilon(t,\eta)
:=
\frac{\mu(t,z^{(0)}+\varepsilon\eta)-\mu(t,z^{(0)})}{\varepsilon},
\qquad
\sigma_0(t):=\sigma(t,z^{(0)}),
\qquad
\sigma_\varepsilon(t,\eta):=\sigma(t,z^{(0)}+\varepsilon\eta).
\]
With this notation, the controlled equations become
\begin{equation} \label{eq: modified dynamics}
	\begin{split}
d\Xi_t^{(p,h,a)}
&=
\Big(
\mu_\varepsilon(t,\Xi^{(p,h,a)})
+\sigma_\varepsilon(t,\Xi^{(p,h,a)})a_t
\Big)dt
+\frac1h\sigma_\varepsilon(t,\Xi^{(p,h,a)})\,dX_t,\\
d\Gamma_t^{(h,a)}
&=
\Big(
A_t\Gamma^{(h,a)}+\sigma_0(t)a_t
\Big)dt
+\frac1h\sigma_0(t)\,dX_t,
\end{split}
\end{equation} 
with common initial value zero. Using $0<\varepsilon\le1$ and \eqref{eq: strong moment}, we find that the coefficients
$(\mu_\varepsilon,\sigma_\varepsilon)$ also satisfy
\ref{B1}--\ref{B2} with a constant independent of $p$ and $h$.
Therefore, as in the proof of
Theorem~\ref{theo-VF-SDE}, these equations admit functional solutions that we denote, with slight abuse of notation, by \(\Xi^{(p, h, a)}\) and \(G/h\).
Moreover, again as in the proof of Theorem~\ref{theo-VF-SDE}, it follows that \(\W\)-a.s. 
\[
\Xi^{(p,h)}\Big(X+h\int_0^\cdot a_t\,dt\Big)
=
\Xi^{(p,h,a)},
\qquad
\frac1hG\Big(X+h\int_0^\cdot a_t\,dt\Big)
=
\frac Gh+v^{(a)}
=
\Gamma^{(h,a)}.
\]
With these observations at hand, the equality in \eqref{eq-app-weighted-gaussian-comparison} follows from Theorem~\ref{theo-VF} applied with the exponent \(\bar{p} = 1 + h^2\), where we also use that \(\sqrt{ \bar{p} - 1 } = h\).

By virtue of the dynamics \eqref{eq: modified dynamics}, the inequality in \eqref{eq-app-weighted-gaussian-comparison} can be proved with similar ideas as used in Steps 4--5 from the proof of Theorem~\ref{thm-app-finite-p-comparison}. First, take \(a \in \mathcal{A}\) such that \(\W\)-a.s. \(\int_0^T \| a_s \|^2 \, ds \leq R\). Then, we get from \eqref{eq-app-controlled-deterministic-error} and \eqref{eq: moment}, applied to \((\mu^\varepsilon, \sigma^\varepsilon)\) with moment order \(\bar{p}\), that 
\begin{align} \label{eq: acc 1}
	\E^\W \Big[ \| \Xi^{(p, h, a)} \|^2_T \Big] \leq c e^{c \sqrt{R}}. 
\end{align}
Next, using the Cauchy--Schwarz inequality, and the Burkholder--Davis--Gundy inequality, we obtain that 
\begin{align*}
    \E^\W \Big[ \| \Xi^{(p, h, a)} - \Gamma^{(h, a)} \|_t \Big] &\leq c\, \Big( \int_0^t \E^\W \Big[ \| \Xi^{(p, h, a)} - \Gamma^{(h, a)} \|_s \Big] \, ds + \frac{h}{\sqrt{p -1}} \int_0^t \E^\W \Big[ \|\Xi^{(p, h, a)}\|^2_s \Big] \, ds \\&\qquad+ \frac{h}{\sqrt{p - 1}} \E^\W \Big[ \| \Xi^{(p, h, a)} \|_t \int_0^t \|a_s\| \, ds\Big] \\&\qquad+ \frac{1}{\sqrt{p - 1}} \E^\W \Big[ \Big(\int_0^t \| \Xi^{(p, h, a)}\|_s^2 \, ds \Big)^{1/2} \Big] \Big)
    \\&\leq c\, \Big( \int_0^t \E^\W \Big[ \| \Xi^{(p, h, a)} - \Gamma^{(h, a)} \|_s \Big] \, ds + \frac{h}{\sqrt{p - 1}} \Big( e^{c \sqrt{R}} + \sqrt{R} e^{c \sqrt{R}} \Big) \Big) 
    \\&\leq c\, \Big( \int_0^t \E^\W \Big[ \| \Xi^{(p, h, a)} - \Gamma^{(h, a)} \|_s \Big] \, ds + \frac{h}{\sqrt{p - 1}} e^{c \sqrt{R}} \Big).
\end{align*}
Now, Gronwall's lemma (that we may apply in view of \eqref{eq: moment} and \eqref{eq: acc 1}) yields that 
\begin{align} \label{eq: acc 2}
	\E^\W \Big[ \| \Xi^{(p, h, a)} - \Gamma^{(h, a)} \|_T \Big] \leq \frac{c h}{\sqrt{p - 1}} e^{c \sqrt{R}}. 
\end{align}
Taking \eqref{eq: acc 2} into account, the inequality in \eqref{eq-app-weighted-gaussian-comparison} follows as in \eqref{eq: summation}.

Finally, we prove \eqref{eq-app-moderate-comparison}.
By virtue of \eqref{eq-app-deterministic-value}, the left-hand side in \eqref{eq-app-moderate-comparison} is bounded from above by 
\begin{align*}
	\sup_{a \in \mathcal{A}} \E^\W\Big[ e^{- \frac{1}{2} \int_0^T \|a_s\|^2 \, ds} \big| \varphi (\Xi^{(p,h,a)}) - \varphi (v^{(a)} ) \big| \Big] \leq \varpi \Big( 	\sup_{a \in \mathcal{A}} \E^\W\Big[ e^{- \frac{1}{2} \int_0^T \|a_s\|^2 \, ds} \|\Xi^{(p,h,a)}- v^{(a)} \|_T \Big] \Big).
\end{align*}
Using that $\Gamma^{(h,a)}-v^{(a)}=G/h$ and $\E^\W[\|G\|_T^2]^{1/2}\le c$, we deduce from \eqref{eq-app-weighted-gaussian-comparison} that 
\begin{align*}
	\sup_{a \in \mathcal{A}} \E^\W\Big[ e^{- \frac{1}{2} \int_0^T \|a_s\|^2 \, ds} \|\Xi^{(p,h,a)}- v^{(a)} \|_T \Big] \leq \frac{c}{h} + \frac{Kh}{\sqrt{p - 1}}, 
\end{align*}
which completes the proof of \eqref{eq-app-moderate-comparison}.
\end{proof}

The first estimate in
Theorem~\ref{thm-app-moderate-comparison} directly compares
$\Xi^{(p,h)}$ with $G/h$. Choosing $h=\sqrt{q-1}$ yields the following
quantitative Gaussian approximation.
\begin{corollary}\label{cor-app-gaussian-approximation}
Suppose that \ref{B1}--\ref{B3} hold. There is a constant $K\ge1$ such
that, for every $p\ge2$ and $2\le q\le p$,
\begin{equation}\label{eq-app-stochastic-taylor}
\E^\W\Big[
\Big\|
\sqrt{p-1}\bigl(Y^{(p)}-z^{(0)}\bigr)-G
\Big\|_T^q
\Big]^{1/q}
\le
\frac{K(q-1)}{\sqrt{p-1}}.
\end{equation}
\end{corollary}

\begin{proof}
Set $h=\sqrt{q-1}$, so that $\bar p=q$. Since
$\sqrt{p-1}(Y^{(p)}-z^{(0)})=h\Xi^{(p,h)}$, multiplying
\eqref{eq-app-weighted-gaussian-comparison} by $h$ proves
\eqref{eq-app-stochastic-taylor}.
\end{proof}

For every fixed $q\ge2$, \eqref{eq-app-stochastic-taylor} implies
convergence to $G$ in $q$-th mean with respect to the uniform norm,
while the case $1\le q<2$ follows from $q=2$. Thus,
\eqref{eq-app-stochastic-taylor} is a quantitative functional central
limit theorem in the uniform topology. The right-hand side still tends
to zero if $q\ge2$ depends on $p$ and $q=o(\sqrt p)$.

Combining this approximation with the Gaussian concentration of $G$,
we obtain the following finite-$p$ concentration inequality for
Lipschitz path functionals.
\begin{corollary}\label{cor:finite-p-concentration}
Suppose that \ref{B1}--\ref{B3} hold. There is a constant $K\ge1$,
depending only on $C$, $T$, and $\|y_0\|$, such that, for every
$p\ge2$, $1\le u\le p-1$, and every $L$-Lipschitz function
$\varphi\colon\Omega\to\bR$ with $L>0$,
\[
\W\Big(
\varphi(Y^{(p)})
-
\E^\W[\varphi(Y^{(p)})]
\ge
KL\Big(
\sqrt{\frac{u}{p-1}}
+
\frac{u}{p-1}
\Big)
\Big)
\le
2e^{-u}.
\]
\end{corollary}

We obtain a lower tail estimate by applying Corollary~\ref{cor:finite-p-concentration} to \(- \varphi\).

\begin{proof}
Throughout the proof, $c$ denotes a finite constant that may change
from line to line and depends only on $C$, $T$, and $\|y_0\|$.

Our strategy is to first obtain a concentration bound for the Gaussian approximation and afterwards to relate it to \(Y^{(p)}\) via Corollary~\ref{cor-app-gaussian-approximation}. 

{\em Step 1.} We first assume that \(\varphi\) is bounded.  
For \(\lambda> 0\) and \(r > 1\), we deduce from Theorem~\ref{theo-VF-SDE} that 
\begin{align*}
\E^\W\Big[ e^{ \lambda \varphi ( z^{(0)} + G / \sqrt{p - 1} ) } \Big]^{1/r} = \sup_{a \in \mathcal{A}} \E^\W\Big[ e^{- \frac{1}{2} \int_0^T \|a_s\|^2 \, ds} e^{ \frac{\lambda}{r} \varphi ( z^{(0)} +  \sqrt{ (r - 1) / (p - 1)}(v^{(a)} + G / \sqrt{ r - 1}) ) } \Big].
\end{align*} 
Using \eqref{eq: strong moment}, \ref{B2} and \ref{B3}, we find that 
\[
\| v^{(a)} \|_t \leq c \Big( \int_0^t \| v^{(a)} \|_s \, ds + \Big( \int_0^t \| a_s \|^2 \, ds \Big)^{1/2} \Big), \quad t \in [0, T].
\]
Hence, by Gronwall's lemma, we obtain 
\[
\| v^{(a)} \|_T \leq c \Big( \int_0^T \|a_s\|^2 \, ds \Big)^{1/2}.
\]
Using this bound and the Lipschitz continuity of \(\varphi\), we obtain 
\begin{align*}
	\E^\W\Big[ &e^{- \frac{1}{2} \int_0^T \|a_s\|^2 \, ds} e^{ \frac{\lambda}{r} \varphi ( z^{(0)} +  \sqrt{ (r - 1) / (p - 1)}(v^{(a)} + G / \sqrt{ r - 1}) ) } \Big] 
	\\&\leq \E^\W\Big[ e^{- \frac{1}{2} \int_0^T \|a_s\|^2 \, ds + \frac{cL\lambda \sqrt{r - 1}}{ r \sqrt{p - 1}} (\int_0^T \| a_s\|^2 \, ds )^{1/2} } e^{ \frac{\lambda}{r} \varphi ( z^{(0)} +  G / \sqrt{p - 1}) } \Big]
	\\&\leq e^{\frac{c^2 L^2 \lambda^2 (r - 1)}{2 r^2 (p - 1)}} \E^\W\Big[ e^{ \frac{\lambda}{r} \varphi ( z^{(0)} +  G / \sqrt{p - 1}) } \Big],
\end{align*}
where we used the standard estimate \(c \sqrt{x} - x / 2 \leq c^2 / 2\) for all \(x \geq 0\).
As the last term is independent of the control \(a\), we conclude that 
\begin{align*}
	\E^\W\Big[ e^{ \lambda \varphi ( z^{(0)} + G / \sqrt{p - 1} ) } \Big]^{1/r} \leq e^{\frac{c^2 L^2 \lambda^2 (r - 1)}{2 r^2 (p - 1)}} \E^\W\Big[ e^{ \frac{\lambda}{r} \varphi ( z^{(0)} +  G / \sqrt{p - 1}) } \Big].
\end{align*}
Taking the \(r\)-th power and letting \(r \to \infty\) yields that  
\[
\E^\W \Big[ e^{ \lambda \varphi ( z^{(0)} + G / \sqrt{p - 1} ) } \Big] \leq e^{\frac{c^2 L^2 \lambda^2}{2 (p - 1)}} e^{\E^\W [ \lambda \varphi ( z^{(0)} + G / \sqrt{p - 1} ) ]}, 
\]
where we used that \(\varphi\) is bounded and standard properties of the moment generating function. 
Although we established this estimate for bounded \(\varphi\), approximating an arbitrary \(\varphi\) by \((-n) \vee \varphi \wedge n\) and using Fatou's lemma, which is possible because of the exponential, we conclude that 
\[
\E^\W \Big[ e^{ \lambda (\varphi ( z^{(0)} + G / \sqrt{p - 1} ) - \E^\W [  \varphi ( z^{(0)} + G / \sqrt{p - 1} ) ] )} \Big] \leq e^{\frac{c^2 L^2 \lambda^2}{2 (p - 1)}}
\]
holds for all \(L\)-Lipschitz functions \(\varphi \colon \Omega \to \bR\).
Now, using Chebyshev's inequality, we obtain 
\[
\W \Big( \varphi ( z^{(0)} + G / \sqrt{p - 1} ) - \E^\W [  \varphi ( z^{(0)} + G / \sqrt{p - 1} ) ] \geq s \Big) \leq e^{- \lambda s + \frac{c^2 L^2 \lambda^2}{2 (p - 1)}}, 
\] 
which holds for all $\lambda>0$. The right-hand side attains its minimum at \(\lambda = \frac{s (p - 1)}{c^2 L^2}\), which implies
\[
\W \Big( \varphi ( z^{(0)} + G / \sqrt{p - 1} ) - \E^\W [  \varphi ( z^{(0)} + G / \sqrt{p - 1} ) ] \geq s \Big) \leq e^{- \frac{s^2 (p - 1)}{2 c^2 L^2}}.
\]
Substituting \(s = c L \sqrt{ 2u/(p - 1)}\), we obtain 
\begin{align*}
    \W \Big( \varphi ( z^{(0)} + G / \sqrt{p - 1} ) - \E^\W [  \varphi ( z^{(0)} + G / \sqrt{p - 1} ) ] \geq c L \sqrt{ \frac{u}{p - 1}} \Big) \leq e^{- u}.
\end{align*}

{\em Step 2.} We next transfer this estimate to $Y^{(p)}$. Taking $q=1+u$ in
\eqref{eq-app-stochastic-taylor} and applying Chebyshev's inequality yields
\[
\W\Big(
\Big\|
Y^{(p)}-z^{(0)}-\frac{G}{\sqrt{p-1}}
\Big\|_T
\ge
\frac{cu}{p-1}
\Big)
\le
e^{-(1+u)}.
\]
The same estimate with $q=2$ also implies
\[
\Big|
\E^\W[\varphi(Y^{(p)})]
-
\E^\W\Big[
\varphi\Big(
z^{(0)}+\frac{G}{\sqrt{p-1}}
\Big)
\Big]
\Big|
\le
\frac{cL}{p-1}.
\]
Consequently,
\begin{align*}
\varphi(Y^{(p)})-\E^\W[\varphi(Y^{(p)})]
&\le
L\Big\|
Y^{(p)}-z^{(0)}-\frac{G}{\sqrt{p-1}}
\Big\|_T\\
&\quad+
\varphi\Big(
z^{(0)}+\frac{G}{\sqrt{p-1}}
\Big)
-
\E^\W\Big[
\varphi\Big(
z^{(0)}+\frac{G}{\sqrt{p-1}}
\Big)
\Big]
+
\frac{cL}{p-1}.
\end{align*}
In summary, for a sufficiently large \(K\), we obtain 
\begin{align*}
	\W \Big( \varphi (Y^{(p)}) - \E^\W [ \varphi (Y^{(p)}) ] \geq K L \Big( \sqrt{ \frac{ u }{ p - 1} } + \frac{u }{p - 1} \Big) \Big) &\leq e^{- ( 1 + u )} + e^{- u} \leq 2 e^{- u}.
\end{align*}
This completes the proof.
\end{proof}

We next turn to moderate deviations. If $h\to\infty$ and
$h/\sqrt{p-1}\to0$, both errors in
\eqref{eq-app-moderate-comparison} vanish, and the deterministic linear
control problem determines the limiting event probabilities. The
following corollary provides the corresponding finite-parameter bounds
and the resulting moderate-deviation principle.
\begin{corollary}\label{cor-app-moderate-deviations}
Suppose that \ref{B1}--\ref{B3} hold, and let $K$ be the constant from
Theorem~\ref{thm-app-moderate-comparison}. Then, for every $p\ge2$,
$1\le h\le\sqrt{p-1}$, $\delta>0$, and every universally measurable
set $B\subseteq\Omega$,
\begin{equation}\label{eq-app-moderate-events}
\begin{aligned}
\Big(
e^{-I_{\mathrm{lin}}(B_{-\delta})}
-\frac K\delta
\Big(
\frac1h+\frac h{\sqrt{p-1}}
\Big)
\Big)_+
&\le
\W\Big(
\frac{\sqrt{p-1}}h
\bigl(Y^{(p)}-z^{(0)}\bigr)\in B
\Big)^{1/(1+h^2)}\\
&\le
e^{-I_{\mathrm{lin}}(B^\delta)}
+\frac K\delta
\Big(
\frac1h+\frac h{\sqrt{p-1}}
\Big).
\end{aligned}
\end{equation}
Take a sequence $(h_p)_{p\ge2}$ that satisfies
$1\le h_p\le\sqrt{p-1}$, $h_p\to\infty$, and
$h_p/\sqrt{p-1}\to0$. Then, the rescaled processes
$\sqrt{p-1}(Y^{(p)}-z^{(0)})/h_p$
satisfy the large-deviation principle on $\Omega$ with speed $h_p^2$
and good rate function $I_{\mathrm{lin}}(\, \cdot\, )$. That means, for
every open set $O\subseteq\Omega$ and every closed set
$F\subseteq\Omega$,
\begin{equation}\label{eq-app-moderate-ldp}
\begin{aligned}
\liminf_{p\to\infty}
\frac1{h_p^2}
\log\W\Big(
\frac{\sqrt{p-1}}{h_p}
\bigl(Y^{(p)}-z^{(0)}\bigr)\in O
\Big)
&\ge
-I_{\mathrm{lin}}(O),\\
\limsup_{p\to\infty}
\frac1{h_p^2}
\log\W\Big(
\frac{\sqrt{p-1}}{h_p}
\bigl(Y^{(p)}-z^{(0)}\bigr)\in F
\Big)
&\le
-I_{\mathrm{lin}}(F).
\end{aligned}
\end{equation}
Consequently, if $B\subseteq\Omega$ is universally measurable and
satisfies
$I_{\mathrm{lin}}(B^\circ)=I_{\mathrm{lin}}(B)
=I_{\mathrm{lin}}(\overline B)<\infty$, then
\begin{equation}\label{eq-app-moderate-continuity-set}
\W\Big(
\frac{\sqrt{p-1}}{h_p} (Y^{(p)}-z^{(0)})
\in
B
\Big)
=
e^{-h_p^2I_{\mathrm{lin}}(B)+o(h_p^2)}
\qquad\text{as }p\to\infty.
\end{equation}
\end{corollary}
\begin{proof}
The bounds in \eqref{eq-app-moderate-events} are immediate when $B=\emptyset$. Otherwise, apply
\eqref{eq-app-moderate-comparison} to the Lipschitz approximations
$\varphi_\delta^-\le\1_B\le\varphi_\delta^+$ used in the proof of
Corollary~\ref{cor-app-event-neighborhoods}. They take values in
$[0,1]$, are $1/\delta$-Lipschitz, and satisfy
$\varphi_\delta^-=1$ on $B_{-\delta}$ and
$\varphi_\delta^+=0$ outside $B^\delta$. Hence, the corresponding
deterministic maxima are at least
$e^{-I_{\mathrm{lin}}(B_{-\delta})}$ and at most
$e^{-I_{\mathrm{lin}}(B^\delta)}$, respectively. This establishes~\eqref{eq-app-moderate-events}.

Now, take $h=h_p$ as described in the corollary. For every fixed $\delta>0$, the error in
\eqref{eq-app-moderate-events} tends to zero and
$(1+h_p^2)/h_p^2\to1$. If
$I_{\mathrm{lin}}(O_{-\delta})<\infty$ and
$I_{\mathrm{lin}}(F^\delta)<\infty$, taking logarithms in the lower
and upper bounds gives
\[
\begin{aligned}
\liminf_{p\to\infty}
\frac1{h_p^2}
\log\W\Big(\Xi^{(p,h_p)}\in O\Big)
&\ge
-I_{\mathrm{lin}}(O_{-\delta}),\\
\limsup_{p\to\infty}
\frac1{h_p^2}
\log\W\Big(\Xi^{(p,h_p)}\in F\Big)
&\le
-I_{\mathrm{lin}}(F^\delta).
\end{aligned}
\]
The lower bound is automatic when
$I_{\mathrm{lin}}(O_{-\delta})=\infty$. If
$I_{\mathrm{lin}}(F^\delta)=\infty$, the upper estimate in
\eqref{eq-app-moderate-events} reduces to its error term. Its logarithm
tends to $-\infty$, so the upper bound remains valid.

Since $O_{-\delta}\uparrow O$ as $\delta\downarrow0$, the definition of
$I_{\mathrm{lin}}(\, \cdot\, )$ implies
$I_{\mathrm{lin}}(O_{-\delta})\to I_{\mathrm{lin}}(O)$.
For closed $F$, the quantities $I_{\mathrm{lin}}(F^\delta)$ increase
to a limit not exceeding $I_{\mathrm{lin}}(F)$. If this limit is
finite, choose $\delta_n\downarrow0$ and
$\eta_n\in F^{\delta_n}$ such that
\[
I_{\mathrm{lin}}(\eta_n)
\le
I_{\mathrm{lin}}(F^{\delta_n})+\frac1n.
\]
Since $I_{\mathrm{lin}}(\, \cdot\, )$ is a good rate function, a subsequence
converges to some $\eta\in\Omega$. The relations
$d_T(\eta_n,F)<\delta_n$ and the closedness of $F$ imply
$\eta\in F$. Lower semicontinuity then shows that
$I_{\mathrm{lin}}(F)$ does not exceed the preceding limit. Consequently,
$I_{\mathrm{lin}}(F^\delta)\to I_{\mathrm{lin}}(F)$, including when
$I_{\mathrm{lin}}(F)=\infty$. Letting $\delta\downarrow0$ proves
\eqref{eq-app-moderate-ldp}.

Finally, if $I_{\mathrm{lin}}(B^\circ)=I_{\mathrm{lin}}(B)
=I_{\mathrm{lin}}(\overline B)<\infty$, applying the lower bound to $B^\circ$ and the upper bound to
$\overline B$ yields
\[
\frac1{h_p^2}
\log\W\Big(\Xi^{(p,h_p)}\in B\Big)
\to
-I_{\mathrm{lin}}(B), \quad p \to \infty.
\]
The definition of $\Xi^{(p,h_p)}$ shows that this is equivalent to
\eqref{eq-app-moderate-continuity-set}.
\end{proof}

\subsection{A stopped variational formula}
\label{subsec-app-stopped-functionals}
We consider functionals that depend only on the path up to a stopping time. Since controls applied after the stopping time cannot change their values, only the control cost accumulated up to that time enters the variational representation.

Recall that $(\cF_t^0)_{t\in[0,T]}$ denotes the raw filtration. 
In this subsection, let $\tau\colon\Omega\to[0,T]$ be an
$(\cF_t^0)_{t\in[0,T]}$-stopping time, and let $\cF_\tau^0$ denote the
corresponding $\sigma$-field. By Galmarino's test
\cite[Lemma~III.2.43]{JS}, every
$\cF_\tau^0$-measurable functional depends only on the path up to $\tau$.

\begin{theorem}\label{thm-app-stopped-functional}
Assume that \ref{A1}--\ref{A3} hold. Let
$\varphi\colon\Omega\to[0,\infty]$ be $\cF_\tau^0$-measurable. Then, for
every $p>1$ and $0<r<s<\infty$,
\begin{equation}\label{eq-app-stopped-multiplicative-identity}
\E^\W\Big[\varphi(Y^{(p)})^s\Big]^{r/s}
=
\sup_{a\in\A}
\E^\W\Big[
e^{
-\frac{r(p-1)}{2(s-r)}
\int_0^{\tau(Z^{(p,a)})}\|a_t\|^2\,dt
}
\varphi(Z^{(p,a)})^r
\Big].
\end{equation}
\end{theorem}

\begin{proof}
A straightforward application of Theorem~\ref{theo-VF-SDE} yields that 
\begin{align*}
    \E^\W\Big[\varphi(Y^{(p)})^s\Big]^{r/s}
=
\sup_{a\in\A}
\E^\W\Big[
e^{
-\frac{r(p-1)}{2(s-r)}
\int_0^{T}\|a_t\|^2\,dt
}
\varphi(Z^{(p,a)})^r
\Big].
\end{align*}
It therefore suffices to show that
\[
\sup_{a\in\A}
\E^\W\Big[
e^{
-\frac{r(p-1)}{2(s-r)}
\int_0^{T}\|a_t\|^2\,dt
}
\varphi(Z^{(p,a)})^r
\Big] \geq \sup_{a\in\A}
\E^\W\Big[
e^{
-\frac{r(p-1)}{2(s-r)}
\int_0^{\tau (Z^{(p, a)})}\|a_t\|^2\,dt
}
\varphi(Z^{(p,a)})^r
\Big], 
\]
as the converse inequality is trivial. To see this, fix an arbitrary \(a \in \mathcal{A}\) and set \(a'_s := a_s \1_{\{s \leq \tau (Z^{(p, a)})\}}\) for \(s \in [0, T]\). Using Galmarino's test \cite[Lemma~III.2.43]{JS} and the Lipschitz conditions \ref{A3}, that entail pathwise uniqueness, it follows that \(\W\)-a.s.
\(
\varphi (Z^{(p, a)}) = \varphi (Z^{(p, a')}).
\)
Hence, we get that
\begin{align*}
\E^\W\Big[
e^{
-\frac{r(p-1)}{2(s-r)}
\int_0^{\tau (Z^{(p, a)})}\|a_t\|^2\,dt
}
\varphi(Z^{(p,a)})^r
\Big] 
&=
\E^\W\Big[
e^{
-\frac{r(p-1)}{2(s-r)}
\int_0^{T}\|a'_t\|^2\,dt
}
\varphi(Z^{(p,a')})^r
\Big]
\\&\leq \sup_{b\in\A}
\E^\W\Big[
e^{
-\frac{r(p-1)}{2(s-r)}
\int_0^{T}\|b_t\|^2\,dt
}
\varphi(Z^{(p,b)})^r
\Big].
\end{align*}
This bound completes the proof. 
\end{proof}

\subsubsection*{Probability of hitting a set}

We now return to the stopped identity and specialize it to sets
$B\in\cF_\tau^0$. By changing the drift, a
control may increase the probability that $Z^{(p,a)}\in B$. Any such
increase comes at a cost, quantified by 
\[
\frac12
\int_0^{\tau(Z^{(p,a)})}\|a_t\|^2\,dt.
\]
The next corollary makes this trade-off precise for every $p>1$ and
every admissible control. The following Theorem~\ref{thm-app-two-scales} shows that
the resulting bound is asymptotically sharp.

\begin{corollary}\label{cor-app-stopped-event-comparison}
Assume that \ref{A1}--\ref{A3} hold and let $B\in\cF_\tau^0$. Then, for every
$p>1$, $a\in\A$, and $R\ge0$,
\begin{equation}\label{eq-app-stopped-event-bound}
\begin{aligned}
&\W\Big(
Z^{(p,a)}\in B,\;
\frac12\int_0^{\tau(Z^{(p,a)})}\|a_t\|^2\,dt\le R
\Big)\\
&\qquad\le
\exp\Big(
-(p-1)
\Big(
\sqrt{
-\tfrac1{p-1}
\log\W\big(Y^{(p)}\in B\big)
}
-\sqrt R
\Big)_+^2
\Big).
\end{aligned}
\end{equation}
\end{corollary}

\begin{proof}
Take \(r \in (0, 1)\).
Using Theorem~\ref{thm-app-stopped-functional} with \(s = 1\) and \(\varphi = \1_B\) yields that 
\begin{align*}
    \E^\W \Big[ e^{- \frac{r (p - 1)}{2 (1 - r)} \int_0^{\tau (Z^{(p, a)})} \| a_t\|^2 \, dt} \1_{\{ Z^{(p, a)} \in B \}} \Big] \leq \W ( Y^{(p)} \in B)^r.
\end{align*}
Using this bound, we immediately get 
\begin{align*}
    \W\Big( Z^{(p,a)}\in B,\; \frac12\int_0^{\tau(Z^{(p,a)})}\|a_t\|^2\,dt\le R \Big) &\leq e^{ \frac{r (p - 1) R}{1 - r}} \W (Y^{(p)} \in B)^r.
\end{align*}
If \(\W (Y^{(p)} \in B) = 0\), this estimate entails the claim. Otherwise, optimizing it over $r\in(0,1)$ and using the identity
$\sup_{0<r<1} (rx-\frac{rR}{1-r})
=(\sqrt{x}-\sqrt R)_+^2$, valid for every $x\ge0$, entails the claimed bound.
\end{proof}

The preceding corollary gives a finite-$p$ bound for each fixed control.
The next theorem combines it with the two-sided comparison in
Corollary~\ref{cor-app-event-neighborhoods} and shows that this bound is
asymptotically sharp after optimization over all controls \(a \in \mathcal{A}\).

\begin{theorem}\label{thm-app-two-scales}
Suppose that \ref{B1}--\ref{B2} hold. Let
$B\in\cF_\tau^0$ be a set that satisfies
$I(B^\circ)=I(B)=I(\overline B)<\infty$. Then, for every $R\ge0$, as
$p\to\infty$,
\begin{equation}\label{eq-app-two-scales}
\begin{aligned}
\sup_{a\in\A}
\W\Big(
&Z^{(p,a)}\in B,\;
\frac12\int_0^{\tau(Z^{(p,a)})}\|a_t\|^2\,dt\le R
\Big)\\
&=
\exp \Big(
-(p-1)
\Big[
\Big(\sqrt{I(B)}-\sqrt R\Big)_+^2+o(1)
\Big]
\Big).
\end{aligned}
\end{equation}
\end{theorem}

\begin{proof}
\smallskip
{\em Step 1.} We first determine the decay rate of the uncontrolled
probability. As in the proof of Corollary~\ref{cor-app-moderate-deviations}, we have $I(B_{-\delta})\to I(B^\circ)$ and \(I (B^\delta) \to I (\overline{B})\).
Using Corollary~\ref{cor-app-event-neighborhoods}, with $p\to\infty$ and then
$\delta\downarrow0$, shows that
$\W(Y^{(p)}\in B)^{1/p}\to\exp(-I(B))$. Since $I(B)<\infty$, this limit
is positive, and therefore
\begin{equation}\label{eq-app-event-rate-limit}
\W(Y^{(p)}\in B)
=
\exp \big(
-(p-1)[I(B)+o(1)]
\big).
\end{equation}

\smallskip
{\em Step 2.} Applying \eqref{eq-app-stopped-event-bound} and taking the
supremum over $a\in\A$ gives
\begin{align*}
\sup_{a\in\A}
\W\Big(
&Z^{(p,a)}\in B,\;
\frac12\int_0^{\tau(Z^{(p,a)})}\|a_t\|^2\,dt\le R
\Big)
\\&\leq
\exp \Big(
-(p-1)
\Big(
\sqrt{
-\tfrac1{p-1}\log\W(Y^{(p)}\in B)
}
-\sqrt R
\Big)_+^2
\Big).
\end{align*}
Using \eqref{eq-app-event-rate-limit} in this estimate proves the upper
bound in \eqref{eq-app-two-scales}.

\smallskip
{\em Step 3.} We prove the reverse bound first when $I(B)>0$. Fix
$\varepsilon>0$. By the definitions of $I(B^\circ)$ and $I(\, \cdot \,)$,
there is a control $\beta\in L^2([0,T];\bR^d)$ such that
$z^{(\beta)}\in B^\circ$ and
\[
\frac12\int_0^T\|\beta_t\|^2\,dt
<
I(B)+\frac{\varepsilon}{2}.
\]
Choose bounded Borel controls $b^n\to\beta$ in $L^2$. Step~1 in the
proof of Theorem~\ref{thm-app-finite-p-comparison} ensures that
$z^{(b^n)}\to z^{(\beta)}$ in $\Omega$, so there exist a bounded
control $b$ and $\delta>0$ such that
\[
z^{(b)}\in B_{-\delta},
\qquad
I(B)
\le
q:=\frac12\int_0^T\|b_t\|^2\,dt
\le
I(B)+\varepsilon.
\]
Since $q>0$, set $\gamma:=\min\{1,\sqrt{R/q}\}$ and $a:=\gamma b$.
Then, 
\begin{align} \label{eq: constraint}
\frac12
\int_0^{\tau(Z^{(p,a)})}\|a_t\|^2\,dt
\le
\gamma^2q
\le R.
\end{align} 
We now regard $Z^{(p,a)}$ as the uncontrolled diffusion \(Y^{(p)}\) with new drift
$(t,\eta)\mapsto \alpha_t (\eta) := \mu(t,\eta)+\gamma\sigma(t,\eta)b_t$ instead of the usual drift \(\mu\). Since $b$ is
bounded, the modified coefficients \((\alpha, \sigma)\) satisfy \ref{B1}--\ref{B2} with
constants independent of $p$. In the corresponding deterministic controlled problem, the control
$(1-\gamma)b$ gives
\[
\alpha(t,\eta)
+\sigma(t,\eta)(1-\gamma)b_t
=
\mu(t,\eta)+\sigma(t,\eta)b_t,
\]
which is precisely the drift defining $z^{(b)}\in B_{-\delta}$.
Hence, 
\[
I_{\alpha, \sigma} (B_{- \delta}) \leq I_{\alpha, \sigma} (z^{(b)}) \leq q(1-\gamma)^2 =  \big(\sqrt q-\sqrt R \big)_+^2, 
\]
where \(I_{\alpha, \sigma}\) denotes the rate function corresponding to the coefficients \((\alpha, \sigma)\).
The lower estimate in
Corollary~\ref{cor-app-event-neighborhoods} therefore yields a constant
$K_b\ge 1$, independent of $p$, such that
\[
\W(Z^{(p,a)}\in B)^{1/p}
\ge
\exp \Big(
-\Big(\sqrt q-\sqrt R\Big)_+^2
\Big)
-
\frac{K_b}{\delta\sqrt{p-1}}.
\]
The right-hand side
converges to
$e^{-(\sqrt q-\sqrt R)_+^2}$ when \(p \to \infty\) and is therefore positive for all sufficiently large $p$.
Using \eqref{eq: constraint}, we obtain 
\begin{align*}
   \frac{1}{p - 1} \log \sup_{\ell\in\A}
\W\Big(
Z^{(p,\ell)}&\in B,\;
\frac12\int_0^{\tau(Z^{(p,\ell)})}\|\ell_t\|^2\,dt\le R
\Big) 
\\&\geq \frac{p}{p - 1} \log \W ( Z^{(p, a)} \in B )^{1/p}
\\&\geq 
 \frac{p}{p - 1} \log \Big( \exp \Big( 
- \Big(\sqrt q-\sqrt R\Big)_+^2
\Big)
- 
\frac{K_b}{\delta\sqrt{p-1}} \Big) 
\\&\geq 
 \frac{p}{p - 1} \log \Big( \exp \Big( 
- \Big(\sqrt{ I (B) + \varepsilon}-\sqrt R\Big)_+^2
\Big)
- 
\frac{K_b}{\delta\sqrt{p-1}} \Big). 
\end{align*}
Letting \(p \to \infty\) and afterwards sending \(\varepsilon \downarrow 0\) yields the exponential lower bound in
\eqref{eq-app-two-scales}.

If $I(B)=0$, the same bound follows directly from the zero control,
since $Z^{(p,0)}=Y^{(p)}$ and \eqref{eq-app-event-rate-limit} applies.
\end{proof}

\section{Proof of Theorem~\ref{theo-VF}} \label{sec: pf VF}

This section is dedicated to the proof of our main Theorem~\ref{theo-VF}. It is sliced into two pieces: in the first we establish a general fact about certain convex expectations on path space, namely that they are determined by their finite-dimensional distributions. This result is based on duality methods and extends a result from \cite{CriensKupper2025} to a framework for convex expectations that only preserve positive constants. We believe that this result is of independent interest. In the second, we establish Theorem~\ref{theo-VF}, combining the convex duality methods with viscosity theory for Hamilton--Jacobi--Bellman PDEs. 

\subsection{Convex Duality and Comparison}

In this section, let \((\Omega, \cF)\) be an arbitrary Polish space with its Borel \(\sigma\)-field and denote the space of all real-valued bounded upper semianalytic functions on \((\Omega, \cF)\) by \(\USA_b (\Omega)\). 

\begin{definition} \label{def: PCP expectation}
    We call a map \(\cE \colon \USA_b (\Omega) \to \bR\) a {\em positive-constant-preserving (PCP) convex expectation} if it satisfies the following properties:
    \begin{enumerate}
        \item[(E1)] \(\cE (c) = c\) for all \(c \in \bR_+\).
        \item[(E2)] \(\cE (\varphi) \leq \cE (\psi)\) for all \(\varphi, \psi \in \USA_b (\Omega)\) with \(\varphi \leq \psi\).
        \item[(E3)] \(\cE (\lambda \varphi + (1 - \lambda) \psi) \leq \lambda \cE (\varphi) + (1 - \lambda) \cE (\psi)\) for all \(\lambda \in [0, 1]\) and \(\varphi, \psi \in \USA_b (\Omega)\). 
        \item[(E4)] \(\cE\) is continuous from below, i.e., for every \(\varphi \in \USA_b (\Omega)\) and any sequence \((\varphi_n)_{n = 1}^\infty \subset \USA_b (\Omega)\) with \(\varphi_n \uparrow \varphi\) it holds that \(\cE (\varphi_n) \uparrow \cE (\varphi)\).
        \item[(E5)] \(\cE\) is continuous from above on \(C_b (\Omega)\), i.e., for every \(\varphi \in \USA_b (\Omega)\) and any sequence \((\varphi_n)_{n = 1}^\infty \subset C_b (\Omega)\) with \(\varphi_n \downarrow \varphi\) it holds that \(\cE (\varphi_n) \downarrow \cE (\varphi)\).
    \end{enumerate}
\end{definition}

Convex expectations admit a dual representation via penalty functions. Let \(\textit{ca}^+ (\Omega)\) be the set of all finite measures on \((\Omega, \cF)\) and endow it with the weak topology. Furthermore, denote its closed subset of sub-probability measures by \(\textit{ca}^+_{\leq 1} (\Omega) = \{ \Q \in \textit{ca}^+ (\Omega) \colon \Q (\Omega) \leq 1 \}\).

\begin{definition}
    A map \(\alpha \colon \textit{ca}^+_{\leq 1} (\Omega) \to [0, \infty]\) is called a {\em penalty function} if it satisfies the following properties:
        \begin{enumerate}
        \item[(P1)] The map \(\Q \mapsto \alpha (\Q)\) is convex.
        \item[(P2)] The level sets \(\{ \Q \in \textit{ca}^+_{\leq 1} (\Omega) \colon \alpha (\Q) \leq c \}\), \(c \in \bR_+\), are compact.
        \item[(P3)] \(\inf_{\Q \in \textit{ca}^+_{\leq 1} (\Omega)} \alpha (\Q) = 0\).
    \end{enumerate}
\end{definition}

\begin{theorem} \label{theo: convex duality; representation}
Let \(\cE\) be a PCP convex expectation. Then, there exists a penalty function \(\alpha\) such that, for all \(\varphi \in \USA_b (\Omega)\),
    \[
    \cE (\varphi) = \sup_{\Q \in \textit{ca}^+_{\leq 1} (\Omega)} \Big( \E^\Q [ \varphi ] - \alpha (\Q) \Big),
    \]
and the penalty function \(\alpha\) admits the dual representation
\[
\alpha (\Q) = \sup_{f \in C_b (\Omega)} \Big( \E^\Q [ f ] - \cE (f) \Big), \quad \Q \in \textit{ca}^+_{\leq 1} (\Omega).
\]
\end{theorem} 

\begin{proof}
By \cite[Theorem~2.2, Remark~2.4]{BartlCheriditoKupper2019}, using the properties (E2)-(E5), 
\[
\cE (\varphi) = \sup_{\Q \in \textit{ca}^+ (\Omega)} \Big( \E^{\Q} [ \varphi ] - \alpha (\Q) \Big), 
\]
with
\[
\alpha (\Q) := \sup_{f \in C_b (\Omega)} \Big( \E^\Q [ f ] - \cE (f) \Big). 
\]
Furthermore, the function \(\alpha\) has compact level sets and it is convex as supremum over affine functions. 
Using (E1), for every \(\Q \in \textit{ca}^+ (\Omega)\) with \(\Q (\Omega) > 1\), and \(c \in \bR_+\), we observe that
\(\alpha (\Q) \geq (\Q (\Omega) - 1) c\), 
which implies that \(\alpha (\Q) = \infty\). Hence, as \(\cE (\varphi)\) is real-valued, 
\[
\cE (\varphi) = \sup_{\Q \in \textit{ca}^+ (\Omega)} \Big( \E^{\Q} [ \varphi ] - \alpha (\Q) \Big) = \sup_{\Q \in \textit{ca}^+_{\leq 1} (\Omega)} \Big( \E^{\Q} [ \varphi ] - \alpha (\Q) \Big).
\]
Using (E1) with \(c = 0\) yields \(\alpha \geq 0\). Furthermore, again with (E1),
\[
1 = \cE (1) = \sup_{\Q \in \textit{ca}^+_{\leq 1} (\Omega)} \Big( \Q (\Omega) - \alpha (\Q) \Big) \leq 1 - \inf_{\Q \in \textit{ca}^+_{\leq 1} (\Omega)} \alpha (\Q), 
\]
which, as \(\alpha \geq 0\), implies that \(\inf_{\Q \in \textit{ca}^+_{\leq 1}(\Omega)} \alpha (\Q) = 0\).
In summary, \(\alpha\) is a penalty function and the proof is complete.
\end{proof}

This dual representation allows us to show that PCP convex expectations on path spaces are determined by their finite dimensional distributions. An abstract version of this result is provided by the next theorem. 

\begin{theorem} \label{theo: comparison} 
Let \(D \subset C_b (\Omega)\) be an \(\bR\)-vector space with \(1 \in D\) that separates \(\textit{ca}^+_{\leq 1} (\Omega)\).\footnote{That is, for any \(\Q, \Q' \in \textit{ca}^+_{\leq 1} (\Omega)\), \(\Q = \Q'\) if and only if \(\E^\Q [ \varphi ] = \E^{\Q'} [ \varphi ]\) for all \(\varphi \in D\).}
For any PCP convex expectations \(\cE\) and \(\widehat{\cE}\), the following are equivalent:
\begin{enumerate}
    \item[(a)] \(\cE \leq \widehat{\cE}\) on \(\USA_b(\Omega)\).
    \item[(b)] \(\cE \leq \widehat{\cE}\) on \(D\). 
\end{enumerate}
\end{theorem} 

\begin{remark} \label{rem: determining set}
    In case \(\Omega = C ([0, T]; E)\) for a Polish space \(E\), a typical choice for \(D\) is given by 
    \[D = \{\omega \mapsto \varphi (\omega (t_1), \dots, \omega (t_n)) \colon n \in \mathbb{N}, \, \varphi \in \textit{Lip}_b (E^n), \, 0 \leq t_1 < \dots < t_n \leq T \};\]
    see, e.g., \cite[Theorem~13.11]{Klenke_20} for the separation property. 
\end{remark}

The proof is based on the following lemma, which is proved below the proof of Theorem~\ref{theo: comparison}.

\begin{lemma} \label{lem: penality identity}
    Let \(D\) be as in Theorem~\ref{theo: comparison} and let \(\alpha\) be a penalty function. Then, for all \(\Q \in \textit{ca}^+_{\leq 1} (\Omega)\), 
    \[
    \alpha (\Q) = \sup_{\varphi \in D} \inf_{\Q' \in \textit{ca}^+_{\leq 1} (\Omega)} \Big( \E^{\Q} [ \varphi ] - \E^{\Q'} [ \varphi ] + \alpha (\Q') \Big). 
    \]
\end{lemma}

\begin{proof}[Proof of Theorem~\ref{theo: comparison}]
Clearly, we only have to prove the implication (b) \(\implies\) (a). Let \(\alpha\) and \(\widehat{\alpha}\) be the penalty functions associated to \(\cE\) and \(\widehat{\cE}\) through Theorem~\ref{theo: convex duality; representation}, respectively. It is clear that \(\cE \leq \widehat{\cE}\) follows from \(\widehat{\alpha} \leq \alpha\). Using Lemma~\ref{lem: penality identity} and the assumption (b), we get, for every \(\Q \in \textit{ca}^+_{\leq 1} (\Omega)\), that
\begin{align*}
    \widehat{\alpha} (\Q) 
    &= \sup_{\varphi \in D} \Big( \E^{\Q} [ \varphi ] - \widehat{\cE} (\varphi) \Big)
    \leq \sup_{\varphi \in D} \Big( \E^{\Q} [ \varphi ] - \cE (\varphi) \Big) = \alpha (\Q). 
\end{align*}
This entails (a) and therefore completes the proof. 
\end{proof}

\begin{proof}[Proof of Lemma~\ref{lem: penality identity}]
As the inequality \(\geq\) is trivial, we only have to establish the \(\leq\) part, assuming that
\[
\sup_{\varphi \in D} \inf_{\Q' \in \textit{ca}^+_{\leq 1} (\Omega)} \Big( \E^{\Q} [ \varphi ] - \E^{\Q'} [ \varphi ] + \alpha (\Q') \Big) =: c < \infty.
\]
For \(N \in \mathbb{N}\), set \(D^N := \{ \varphi \in D \colon \|\varphi\|_\infty \leq N \}\) and \(\mathcal{P}_N := \{ \Q' \in \textit{ca}^+_{\leq 1} (\Omega) \colon \alpha (\Q') \leq 2N + 1 \}\). 
For a function \(\varphi \in D^N\), using the property (P3), it follows readily that
\begin{align*} 
\inf_{\Q' \in \textit{ca}^+_{\leq 1} (\Omega)} \Big( \alpha (\Q') - \E^{\Q'} [ \varphi ] \Big) = \inf_{\Q' \in \mathcal{P}_N} \Big( \alpha (\Q              ') - \E^{\Q'} [ \varphi ] \Big).
\end{align*} 
By this identity, we get that 
\begin{align*} 
c = \liminf_{N \to \infty} \sup_{\varphi \in D^N} \inf_{\Q' \in \cP_N} \Big( \E^{\Q} [ \varphi ] - \E^{\Q'} [ \varphi ] + \alpha (\Q') \Big).
\end{align*}
Property (P2) implies that \(\cP_{N}\) is compact and that
\(
\Q' \mapsto - \E^{\Q'} [ \varphi ] + \alpha (\Q')
\)
is lower semicontinuous. Furthermore, it is convex by property (P1). As 
\(
\varphi \mapsto \E^{\Q} [ \varphi ] - \E^{\Q'} [ \varphi ] 
\)
is concave, Fan's minimax theorem \cite[Theorem~2]{fan} yields that
\begin{align*} 
\sup_{\varphi \in D^N} \inf_{\Q' \in \cP_N} & \Big( \E^{\Q} [ \varphi ] - \E^{\Q'} [ \varphi ] + \alpha (\Q') \Big)
\\&= \inf_{\Q' \in \cP_N} \sup_{\varphi \in D^N} \Big( \E^{\Q} [ \varphi ] - \E^{\Q'} [ \varphi ] + \alpha (\Q') \Big).
\end{align*}
Take a measure \(\Q^N \in \cP_N\) such that 
\begin{align*}
    \inf_{\Q' \in \cP_N} \sup_{\varphi \in D^N} &\Big( \E^{\Q} [ \varphi ] - \E^{\Q'} [ \varphi ] + \alpha (\Q') \Big) 
    \\&\geq \sup_{\varphi \in D^N} \Big( \E^{\Q} [ \varphi ] - \E^{\Q^N} [ \varphi ] + \alpha (\Q^N) \Big) - \frac{1}{N}
\end{align*}
Passing to a suitable subsequence, we conclude that 
\begin{align*}
    c \geq \lim_{n \to \infty} \sup_{\varphi \in D^{N_n}} \Big( \E^{\Q} [ \varphi ] - \E^{\Q^{N_n}} [ \varphi ] + \alpha (\Q^{N_n}) \Big).
\end{align*}
Without loss of generality, we may assume that for all \(n \in \mathbb{N}\) 
\[
\sup_{\varphi \in D^{N_n}} \Big( \E^{\Q} [ \varphi ] - \E^{\Q^{N_n}} [ \varphi ] \Big) + \alpha (\Q^{N_n})  \leq c + 1.
\]
Hence, as 
\[
0 \in D^{N_n} \implies \sup_{\varphi \in D^{N_n}} \Big( \E^{\Q} [ \varphi ] - \E^{\Q^{N_n}} [ \varphi ] \Big) \geq 0, 
\]
we conclude that \(\alpha (\Q^{N_n}) \leq c + 1\). Thus, these measures belong to the compact set
$\{\Q'\in\textit{ca}^+_{\leq1}(\Omega)
\mathrel{\colon}\alpha(\Q')\le c+1\}$.
Passing to a subsequence if necessary, they converge weakly to a
measure $\Q^*$ in this set.
Therefore, using that \(\alpha\) is lower semicontinuous (which follows from (P2)), we obtain that 
\begin{align}
c &\geq \limsup_{n \to \infty} \sup_{\varphi \in D^{N_n}} \Big( \E^{\Q} [ \varphi ] - \E^{\Q^{N_n}} [ \varphi ] \Big) + \alpha (\Q^*) \label{eq: last step}
\\&= \limsup_{n \to \infty} \sup_{\varphi \in D^{1}} N_n \Big( \E^{\Q} [ \varphi ] - \E^{\Q^{N_n}} [ \varphi ] \Big) + \alpha (\Q^*). \nonumber
\end{align} 
Here, we use that \(N D^1 = D^N\), as \(D\) is a vector space.
This implies that 
\[
\lim_{n \to \infty} \sup_{\varphi \in D^{1}} \Big( \E^{\Q} [ \varphi ] - \E^{\Q^{N_n}} [ \varphi ] \Big) = 0.
\]
In particular, for all \(\varphi \in D^1\), we have 
\[
\E^{\Q} [ \varphi ] = \lim_{n \to \infty} \E^{\Q^{N_n}}[ \varphi ] = \E^{\Q^*} [ \varphi ], 
\]
where we use that \(\varphi \in D^1\) implies that also \(- \varphi \in D^1\).
As the set \(D^1\) is separating for \(\textit{ca}^+_{\leq 1} (\Omega)\), we conclude that \(\Q = \Q^*\).
Finally, we get from \eqref{eq: last step} that 
\(
\alpha (\Q) = \alpha (\Q^*) \leq c, 
\)
which completes the proof.
\end{proof}

\subsection{Connecting Theorems~\ref{theo-VF} and \ref{theo: comparison}}

We now connect the variational formula~\eqref{eq-main-VF} to the concept of PCP expectations from Definition~\ref{def: PCP expectation}. From now on, \((\Omega, \cF)\) is as in Section~\ref{sec-variational-formulas}, i.e., the space \(C ([0, T]; \bR^d)\) with its Borel \(\sigma\)-field.
For \(\varphi \in \USA_b (\Omega)\), we set 
\begin{equation} \label{eq: main PCP}
\begin{split}
    \mathcal{E} (\varphi) &:= \E^{\W} \big[ \varphi^+ (X)^p \big]^{1/p}, \\
    \widehat{\cE} (\varphi) &:= \sup_{a \in \A} \E^{\W} \Big[ e^{- \frac{1}{2} \int_0^T \| a_s \|^2 \, ds } \, \varphi^+ \Big (X + \sqrt{p - 1} \int_0^\cdot a_s \, ds \Big) \Big],
\end{split}
\end{equation} 
where \(\varphi^+ := \max \{ \varphi, 0\}\).

\begin{lemma} \label{lem: main are PCP}
    Both \(\cE\) and \(\widehat{\cE}\) are PCP expectations (in the sense of Definition~\ref{def: PCP expectation}).
\end{lemma}

The only non-obvious part in the proof of Lemma~\ref{lem: main are PCP} is the property (E5) for \(\widehat{\cE}\). We prepare its proof with the following general lemma, which is of interest in its own. Namely, besides establishing (E5), it already gives a quite easy proof for the inequality \(\geq\) from the variational representation \eqref{eq-main-VF}. Moreover, we use it to establish Corollary~\ref{coro-arbitrary-setup}.

\begin{lemma} \label{lem: geq inequality in main}
    Let $(\widetilde{\Omega},\widetilde{\cF},(\widetilde{\cF}_t)_{t \in [0, T]},\widetilde{\P})$ be a filtered probability space with completed filtration that supports a \(d\)-dimensional $(\widetilde{\cF}_t)_{t \in [0, T]}$-Brownian motion $\widetilde{W} = (\widetilde{W}_t)_{t \in [0, T]}$. Furthermore, let \((a_t)_{t \in [0, T]}\) be an \(\bR^d\)-valued \((\widetilde{\cF}_t)_{t \in [0, T]}\)-progressively measurable process such that \(\widetilde{\P}\)-a.s.
    \[
    \int_0^T \| a_s \|^2 \, ds < \infty. 
    \]
    Then, for every non-negative universally measurable function \(\varphi \colon \Omega \to \bR_+\) and every \(p \geq 1\), 
    \[
    \E^{\widetilde{\P}} \Big[ e^{- \frac{1}{2} \int_0^T \|a_s\|^2 \,ds} \, \varphi \Big( \widetilde{W} + \sqrt{p - 1}\, \int_0^\cdot a_s \, ds\Big) \Big] \leq \E^{\W} \Big[ \varphi (X)^p \Big]^{1/p}.
    \]
\end{lemma}
\begin{proof}
As the case \(p=1\) is trivial, from now on, we assume that \(p>1\).
Moreover, fix \(m \in \mathbb{N}\) and define the \((\widetilde{\cF}_t)_{t \in [0, T]}\)-stopping time 
  \[
\tau_m := \inf \Big \{ t \in [0, T] \colon \int_0^t \| a_s \|^2 \, ds \geq m \Big\}, 
\]
where \(\inf \emptyset := \infty\).
By Novikov's condition (\cite[Proposition~3.5.12]{KaratzasShreve1991}), the process
\[
Z^m := \exp \Big(- \int_0^{\cdot \wedge \tau_m} \sqrt{p - 1} a_s \, d \widetilde{W}_s - \frac{(p - 1)}{2} \int_0^{\cdot \wedge \tau_m} \| a_s \|^2 \, ds \Big)
\]
is a \(\widetilde{\P}\)-martingale and consequently, we may define a probability measure \(\widetilde{\Q}_m\) on \((\widetilde{\Omega}, \widetilde{\cF})\) through the Radon--Nikodym derivative 
\[
\frac{d \widetilde{\Q}_m}{d \widetilde{\P}} = Z^m_T. 
\]
By Girsanov's theorem (\cite[Theorem~3.5.1]{KaratzasShreve1991}), the process \(\widetilde{W} + \sqrt{p - 1}  \int_0^{\cdot \wedge \tau_m} a_s \, ds\) is a \(\widetilde{\Q}_m\)-Brownian motion. Using the monotone convergence theorem, we get that 
\begin{align*}
    \E^{\widetilde{\P}}\Big[ &e^{- \frac{1}{2} \int_0^T \|a_s\|^2 \, ds} \varphi \Big( \widetilde{W} + \sqrt{p - 1} \, \int_0^\cdot a_s \, ds \Big) \Big] 
    \\&= \lim_{m \to \infty} \E^{\widetilde{\P}}\Big[ e^{- \frac{1}{2} \int_0^T \|a_s\|^2 \, ds} \varphi \Big( \widetilde{W} + \sqrt{p - 1} \, \int_0^\cdot a_s \, ds \Big) \1_{\{\tau_m \geq T \}} \Big].
\end{align*}
Moreover, by H\"older's inequality,
\begin{align*}
       \E^{\widetilde{\P}} &\Big[ e^{- \frac{1}{2} \int_0^{T \wedge \tau_m} \|a_s\|^2 \, ds} \varphi \Big( \widetilde{W} + \sqrt{p - 1} \, \int_0^\cdot a_s \, ds \Big) \1_{\{\tau_m \geq T \}} \Big] 
    \\&= \E^{\widetilde{\Q}_m}\Big[ e^{\int_0^{T \wedge \tau_m} \sqrt{p - 1} a_s (d\widetilde{W}_s + \sqrt{p - 1} a_s \, ds) - \frac{p}{2} \int_0^{T \wedge \tau_m} \|a_s\|^2 \, ds} \varphi \Big( \widetilde{W} + \sqrt{p - 1} \, \int_0^\cdot a_s \, ds \Big) \1_{\{\tau_m \geq T \}}\Big] 
    \\&\leq \E^{\widetilde{\Q}_m}\Big[ e^{\frac{p}{\sqrt{p - 1}} \, \int_0^{T \wedge \tau_m} a_s (d \widetilde{W}_s + \sqrt{p - 1} a_s \, ds) - \frac{p^2}{2 (p - 1)} \int_0^{T \wedge \tau_m} \|a_s\|^2 \, ds} \Big]^{(p - 1)/ p} \\&\hspace{5cm}\E^{\widetilde{\Q}_m} \Big[ \varphi \Big( \widetilde{W} + \sqrt{p - 1} \, \int_0^\cdot a_s \, ds \Big)^p \1_{\{\tau_m \geq T \}} \Big]^{1/p}
    \\&\leq \E^{\widetilde{\Q}_m} \Big[ \varphi \Big( \widetilde{W} + \sqrt{p - 1} \, \int_0^{\cdot \wedge \tau_m} a_s \, ds \Big)^p \Big]^{1/p},
\end{align*}
where we used that
\[
\exp \Big(\frac{p}{\sqrt{p - 1}} \, \int_0^{\cdot \wedge \tau_m} a_s \big(d \widetilde{W}_s + \sqrt{p - 1} a_s \, ds\big) - \frac{p^2}{2 (p - 1)} \int_0^{\cdot \wedge \tau_m} \|a_s\|^2 \, ds \Big)
\]
is a \(\widetilde{\Q}_m\)-martingale by Novikov's condition.
Summing up, using once again that \(\widetilde{W} + \sqrt{p - 1} \int_0^{\cdot \wedge \tau_m} a_s \, ds\) is a \(\widetilde{\Q}_m\)-Brownian motion, we conclude that 
\begin{align*}
    \E^{\widetilde{\P}}\Big[ &e^{- \frac{1}{2} \int_0^T \|a_s\|^2 \, ds} \varphi \Big( \widetilde{W} + \sqrt{p - 1} \, \int_0^\cdot a_s \, ds \Big) \Big] 
    \\&\leq  \limsup_{m \to \infty} \E^{\widetilde{\Q}_m} \Big[ \varphi \Big( \widetilde{W} + \sqrt{p - 1} \, \int_0^{\cdot \wedge \tau_m} a_s \, ds \Big)^p \Big]^{1/p}
    = \E^{\W} \big[ \varphi (X)^p \big]^{1/p}. \qedhere
\end{align*} 
\end{proof}

\begin{proof}[Proof of Lemma~\ref{lem: main are PCP}]
    We start with \(\cE\). Obviously, (E1) and (E2) hold. Property (E3) follows from the convexity of \(x \mapsto x^+\) and the Minkowski inequality. Finally, (E4) and (E5) are consequences of the monotone/dominated convergence theorems. 

We turn to \(\widehat{\cE}\). In this case (E2) is obvious and (E3) follows from the convexity of \(x \mapsto x^+\). Property (E1) is due to the fact that 
\[
\sup_{a \in \mathcal{A}} \E^{\W} \Big[ e^{- \frac{1}{2} \int_0^T \| a_s \|^2 \, d s} \Big] = 1.
\]
The continuity property (E4) follows from the monotone convergence theorem, switching two suprema. 
Finally, we discuss (E5). Let \(\varphi \in \USA_b (\Omega)\) and take a sequence \((\varphi_n)_{n = 1}^\infty \subset C_b (\Omega)\) such that \(\varphi_n \downarrow \varphi\). 
In view of Lemma~\ref{lem: geq inequality in main}, we obtain that 
\begin{align*}
    \big| \widehat{\cE} (\varphi_n) - \widehat{\cE} (\varphi) \big| &= \widehat{\cE} (\varphi_n) - \widehat{\cE} (\varphi)
    \leq \widehat{\cE} (\varphi_n - \varphi)
    \leq \E^{\W} \big[ (\varphi_n - \varphi)^p \big]^{1/p} \to 0, \quad n \to \infty, 
\end{align*}
by the dominated convergence theorem (using that \(0 \leq \varphi^+_n - \varphi^+ \leq \varphi_n - \varphi \leq \|\varphi_1\|_\infty + \|\varphi\|_\infty\)). This shows that \(\widehat{\cE}\) satisfies the property (E5). \end{proof}

\subsection{Proof of Theorem~\ref{theo-VF}} \label{sec: pf VF relaxed}
As the case \(p = 1\) is trivial, we assume \(p > 1\) in the following.

\smallskip 
{\em Step 1:} 
We first explain that it suffices to prove the variational formula for all non-negative \(\varphi \in \USA_b (\Omega)\). The reduction to bounded functions follows from the monotone convergence theorem. To reduce the claim to upper semianalytic functions, assume that \eqref{eq-main-VF} holds for all non-negative bounded upper semianalytic functions and take a non-negative bounded universally measurable function \(\varphi \colon \Omega \to \bR_+\). In view of \cite[Lemma~7.27]{bershre}, there exists a non-negative bounded \(\cF\)-measurable function \(\psi \colon \Omega \to \bR_+\) such that \(\W\)-a.s. \(\varphi = \psi\). As Borel functions are upper semianalytic, the variational formula holds for \(\psi\) by assumption, and we obtain 
\[
\E^\W \big[ \varphi (X)^p \big]^{1/p} = \E^\W \big[ \psi (X)^p \big]^{1/p} = \sup_{a \in \mathcal{A}} \E^\W \Big[ e^{- \frac{1}{2} \int_0^T \|a_s\|^2 \, ds} \psi \Big( X + \sqrt{p - 1} \int_0^\cdot a_s \, ds \Big) \Big]. 
\]
In view of \cite[Theorem~7.2]{LS_01}, for every \(a \in \mathcal{A}\), \(\W \circ (X + \sqrt{p - 1}\,\int_0^\cdot a_s \, ds)^{-1} \ll \W\) and consequently, \(\W\)-a.s. 
\[
\varphi \Big( X + \sqrt{p - 1} \int_0^\cdot a_s \, ds \Big) = \psi \Big( X + \sqrt{p - 1}\int_0^\cdot a_s \, ds \Big). 
\]
This entails that 
\begin{align*}
\sup_{a \in \mathcal{A}} \E^\W &\Big[ e^{- \frac{1}{2} \int_0^T \|a_s\|^2 \, ds} \psi \Big( X + \sqrt{p - 1} \int_0^\cdot a_s \, ds \Big) \Big] 
\\&= \sup_{a \in \mathcal{A}} \E^\W \Big[ e^{- \frac{1}{2} \int_0^T \|a_s\|^2 \, ds} \varphi \Big( X + \sqrt{p - 1} \int_0^\cdot a_s \, ds \Big) \Big], 
\end{align*}
showing that the variational formula also holds for \(\varphi\).

{\em Step 2:} We now reduce the problem further, showing that it suffices to consider finite-dimensional marginals. 
Let \(\cE\) and \(\widehat{\cE}\) be as in \eqref{eq: main PCP} and recall from Lemma~\ref{lem: main are PCP} that both are PCP expectations. 
Thus, by virtue of the comparison Theorem~\ref{theo: comparison}, to establish the variational formula~\eqref{eq-main-VF} for all non-negative \(\varphi \in \USA_b (\Omega)\), it suffices to establish it for all non-negative finite-dimensional distributions from Remark~\ref{rem: determining set}. In fact, using the continuity from above property (E5), we may even assume that the test functions are bounded away from zero. More precisely, it suffices to prove \eqref{eq-main-VF} for 
\[
\omega \mapsto \varphi (\omega) \equiv f (\omega (t_1), \dots, \omega (t_n)), 
\]
where \(n \in \mathbb{N}\), \(f \in \textit{Lip}_b (\bR^{dn})\) with \(\inf f > 0\), and \(0 \leq t_1 < \dots < t_n \leq T\). We denote the class of these test functions by \(D^+\).

{\em Step 3:} We now prove \eqref{eq-main-VF} for all finite-dimensional marginals. 
For \(x \in \bR^d, t \in [0, T]\) and all non-negative \(f \in \textit{Lip}_b (\bR^d)\), we set 
\begin{align*}
S_t [f] (x) &:= \cE ( f (x + X_t) ), \qquad
\widehat{S}_t [f] (x) := \widehat{\cE} ( f (x + X_t) ).
\end{align*}
For \(n \geq 2\), all non-negative \(f \in \textit{Lip}_b (\bR^{dn})\) and \(0 \leq t_1 < \dots < t_n \leq T\), it holds that
\begin{align} \label{eq: fundamental reduction identity 1}
\cE (f (X_{t_1}, \dots, X_{t_n}) ) &= \cE \Big( S_{t_n - t_{n - 1}} [f (x_1, \dots, x_{n - 1}, \cdot \,)] (x_{n - 1}) \, \big|_{x_1 = X_{t_1}, \dots, x_{n - 1} = X_{t_{n - 1}}} \Big).  
\end{align}
To ease our presentation, we prove this identity only for \(n = 2\). Starting with the right-hand side and using the Markov property of Brownian motion, we obtain that
\begin{align*} 
\cE \big( S_{t_2 - t_1} [ f (x_1, \cdot \,)] (x_1) \, \big|_{x_1 = X_{t_1}} \big) 
&= \E^{\W} \Big[ \E^{\W} [ f (x_1, x_1 + X_{t_2 - t_1})^p] \, \big|_{x_1 = X_{t_1}} \Big]^{1 / p}
\\&= \E^{\W} \Big[ \E^{\W} \big[ f (X_{t_1}, X_{t_2})^p \mid \cF_{t_1} \big] \Big]^{1/p} 
= \cE ( f (X_{t_1}, X_{t_2}) ). \phantom \int
\end{align*} 
Thanks to Lemma~\ref{lem: coincide} and Corollary~\ref{coro: for proof}, a similar claim holds for \(\widehat{\cE}\), namely
\begin{align} \label{eq: fundamental reduction identity 2}
\widehat{\cE} (f (X_{t_1}, \dots, X_{t_n}) ) &= \widehat{\cE} \Big( \widehat{S}_{t_n - t_{n - 1}} [f (x_1, \dots, x_{n - 1}, \cdot \,)] (x_{n - 1}) \, \big|_{x_1 = X_{t_1}, \dots, x_{n - 1} = X_{t_{n - 1}}} \Big).  
\end{align}
By induction, \eqref{eq: fundamental reduction identity 1} and \eqref{eq: fundamental reduction identity 2}, the property \(\cE = \widehat{\cE}\) on \(D^+\) follows once we prove the following two properties:
\begin{enumerate} 
\item[(a)] For all non-negative \(f \in \textit{Lip}_b (\bR^{dn})\), the map \((x_1, \dots, x_{n - 1}) \mapsto S_t [ f (x_1, \dots, x_{n-1}, \cdot \,)] (x_{n - 1})\) belongs to \(\textit{Lip}_b (\bR^{d (n - 1)})\).
\item[(b)] \(S_t [ f ] (x) = \widehat{S}_t [ f ] (x)\) for all \(x \in \bR^d, t \in [0, T],\) and \(f \in \textit{Lip}_b (\bR^d)\) with \(\inf f > 0\).
\end{enumerate} 
Before we show that these items hold, let us shortly outline the induction argument. The induction base holds by (b). Now, assume as induction hypothesis that 
\[\cE = \widehat{\cE} \text{ on } \big\{ f (X_{t_1}, \dots, X_{t_{n - 1}}) \colon 0 \leq t_1 < \dots < t_{n - 1} \leq T, \, f \in \textit{Lip}_b(\bR^{d(n - 1)}) \text{ with } \inf f > 0 \big\}
\]
for a given \(n \geq 2\). Then, for every \(f \in \textit{Lip}_b (\bR^{dn})\) with \(\inf f > 0\) and \(0 \leq t_1 < \dots < t_{n} \leq T\), using \eqref{eq: fundamental reduction identity 1} for the first identity, the induction hypothesis and (a) for the second identity, (b) for the third identity, and \eqref{eq: fundamental reduction identity 2} for the last identity, we obtain that
\begin{align*}
    \cE \big( f (X_{t_1}, \dots, X_{t_{n}}) \big) &= \cE \Big( S_{t_n - t_{n - 1}} [f (x_1, \dots, x_{n - 1}, \cdot \,)] (x_{n - 1}) \, \big|_{x_1 = X_{t_1}, \dots, x_{n - 1} = X_{t_{n - 1}}} \Big)
    \\&= \widehat{\cE} \Big( S_{t_n - t_{n - 1}} [f (x_1, \dots, x_{n - 1}, \cdot \,)] (x_{n - 1}) \, \big|_{x_1 = X_{t_1}, \dots, x_{n - 1} = X_{t_{n - 1}}} \Big)
    \\&= \widehat{\cE} \Big( \widehat{S}_{t_n - t_{n - 1}} [f (x_1, \dots, x_{n - 1}, \cdot \,)] (x_{n - 1}) \, \big|_{x_1 = X_{t_1}, \dots, x_{n - 1} = X_{t_{n - 1}}} \Big)
    \\&= \widehat{\cE} \big( f (X_{t_1}, \dots, X_{t_{n}}) \big).
\end{align*}
Hence, to conclude \(\cE = \widehat{\cE}\) it suffices to show that (a) and (b) hold.

We start with (a), restricting ourselves to the case \(n = 2\) for simplicity.
Using the \(1\)-Lipschitz continuity of the \(L^p\)-norm, we obtain that 
\begin{align*}
    \big| S_t [ f (x_1, \cdot \,)] (x_2) &- S_t [ f (y_1, \cdot \,)] (y_2) \big| 
    \leq L \big( \|x_1 - y_1\| + \|x_2 - y_2\| \big), 
\end{align*}
where \(L\) denotes the Lipschitz constant of \(f\).
For (b), take \(f \in \textit{Lip}_b (\bR^{d})\) with \(\inf f > 0\).
The function
\[
u (s, x) := \E^{\W} \big[ f (x + X_{t - s})^p \big], \quad (s, x) \in [0, t] \times \bR^d, 
\]
is a (classical) solution to the backward heat equation (see \cite[Section~4.3]{KaratzasShreve1991}). In particular, we notice that \(u\) is uniformly bounded away from zero by \((\inf f )^p\). Using the transformation \(y \mapsto y^{1 / p}\), and standard rules on a change of variable for viscosity solutions (see, e.g., \cite[Proposition~6.7]{T_13} for a result in this direction), we obtain that \(v := u^{1/p}\) is a viscosity solution to the nonlinear PDE
\begin{align*}
    - \frac{\partial v}{\partial t} - \frac{1}{2} (p - 1) \frac{\|Dv\|^2}{v} - \frac{1}{2} \on{tr} \big[ D^2 v \big] = 0.
\end{align*}
Noting that 
\[
\frac{1}{2} (p - 1) \frac{\|Dv\|^2}{v} = \sup_{a \in \bR^d} \Big[ \sqrt{p - 1} \langle Dv, a \rangle - \frac{1}{2} v \| a \|^2  \Big], 
\]
with the unique expression $a^*=\sqrt{p - 1}\,Dv/v=\sqrt{p - 1}\,D\log v$ that maximizes the right hand side, 
we conclude that \(v\) is a viscosity solution to the Hamilton--Jacobi--Bellman PDE 
\begin{align} \label{eq: viscosity}
    - \frac{\partial v}{\partial t} - \sup_{a \in \bR^d} \Big[ \sqrt{p - 1} \langle D v, a \rangle - \frac{1}{2} v \| a \|^2 \Big] - \frac{1}{2} \on{tr} \big[ D^2 v \big] = 0.
\end{align}
Consider the function
\[
\widehat{v} (s, x) := \widehat{S}_{t - s} [ f ] (x), \quad (s, x) \in [0, t] \times \bR^d. 
\]
We now show that \(\widehat{v}\) is continuous. If \(L\) denotes the Lipschitz constant of \(f\), we obtain, for all \(s, u \in [0, t]\) and \(x, y \in \bR^d\), that
\begin{align*}
    \big| \widehat{v} (s&, x) - \widehat{v} (u, y) \big| 
    \\&\leq \sup_{a \in \mathcal{A}} \E^\W \Big[ e^{- \frac{1}{2} \int_0^T \| a_z \|^2 \, dz} \Big| f \Big(x + X_{t - s} + \sqrt{p - 1} \, \int_0^{t - s} a_z \, dz \Big) \\&\hspace{4cm}- f \Big( y + X_{t - u} + \sqrt{p - 1}\, \int_0^{t - u} a_z \, dz \Big) \Big| \Big]
    \\&\leq L \Big( \|x - y\| + \E^\W \Big[ \| X_{t - s} - X_{t - u} \| \Big] 
    \\&\hspace{4cm}+ \sqrt{p - 1}\,\sup_{a \in \mathcal{A}} \E^\W \Big[ e^{- \frac{1}{2} \int_0^T \|a_z\|^2 \, dz} \Big\| \int_{t - (s \vee u)}^{t - (s \wedge u)} a_z \, dz \Big\| \Big] \Big)
    \\&\leq L \Big( \|x - y\| + |s - u|^{1/2} \Big( \sqrt{d} + \sqrt{p - 1}\, \sup_{a \in \mathcal{A}} \E^\W \Big[ e^{- \frac{1}{2}\int_0^T \|a_z\|^2 \, dz} \Big( \int_0^T \|a_z\|^2 \, dz \Big)^{1/2}\Big] \Big) \Big)
    \\&\leq L \Big( \|x - y\| + | s - u|^{1/2} \Big( \sqrt{d} + \sqrt{\frac{p - 1}{e}}\Big) \Big), 
\end{align*}
where we used that \([0, \infty) \ni \ell \mapsto e^{-\ell / 2} \sqrt{\ell}\) is maximized at \(\ell = 1\). This shows that \(\widehat{v}\) is continuous. 
We now show that \(\widehat{v}\) satisfies viscosity properties associated to the nonlinear PDE \eqref{eq: viscosity}, starting with the subsolution part.
Take \((s, y) \in [0, t) \times \bR^d\) and a function \(\psi \in C^\infty_b ([0, t] \times \bR^d)\) such that \(\widehat{v} \leq \psi\) and \(\psi (s, y) = \widehat{v} (s, y)\). Notice that \(\psi\) is bounded away from zero by \(\inf f\).
We aim to show that 
\[
- \frac{\partial \psi}{\partial t} (s, y) - \sup_{a \in \bR^d} \Big[ \sqrt{p - 1} \langle D \psi (s, y), a \rangle - \frac{1}{2} \psi (s, y) \| a \|^2 \Big] - \frac{1}{2} \on{tr} \big[ D^2 \psi (s, y)\big] \leq 0.
\]
From Lemma~\ref{lem: coincide} and Corollary~\ref{coro: for proof}, we obtain, for every \(u \in (0, t - s)\), that 
\begin{equation} \label{eq: DPP application} \begin{split} 
    0 &=
    \sup_{a \in \mathcal{A}} \E^\W \Big[ e^{- \frac{1}{2} \int_0^{T} \|a_z\|^2 \, dz} \, \widehat{v} \Big(s + u, y + X_{u} + \sqrt{p - 1} \, \int_0^u a_z \, dz\Big) - \widehat{v} (s, y) \Big]
    \\&= 
    \sup_{a \in \mathcal{A}} \E^\W \Big[ e^{- \frac{1}{2} \int_0^{u} \|a_z\|^2 \, dz} \, \widehat{v} \Big(s + u, y + X_{u} + \sqrt{p - 1} \, \int_0^u a_z \, dz\Big) - \widehat{v} (s, y) \Big], 
\end{split} \end{equation} 
where the last equality follows from the fact that the supremum always takes \(a \equiv 0\) on \((u, T]\).
Clearly, the first supremum ignores processes that are not contained in 
\[
\mathcal{A}^* := \Big\{ a \in \mathcal{A} \colon \E^\W \Big[ e^{ - \frac{1}{2} \int_0^T \|a_z\|^2 \, dz } \Big] \geq \frac{1}{2}\frac{\inf f}{\sup f} \Big\}.  
\]
Thus, in both suprema from \eqref{eq: DPP application} we may replace \(\mathcal{A}\) by \(\mathcal{A}^*\), getting that
\begin{align*} 
    0 = \sup_{a \in \mathcal{A}^*} \E^\W \Big[ e^{- \frac{1}{2} \int_0^{u} \|a_z\|^2 \, dz} \, \widehat{v} \Big(s + u, y + X_{u} + \sqrt{p - 1} \, \int_0^u a_z \, dz\Big) - \widehat{v} (s, y) \Big].
\end{align*}
Using the properties of the test function \(\psi\), we further obtain that 
\begin{align*}
    0 &\leq \sup_{a \in \mathcal{A}^*} \E^\W \Big[ e^{- \frac{1}{2} \int_0^u \|a_z\|^2 \, dz} \, \psi \Big(s + u, y + X_{u} + \sqrt{p - 1} \, \int_0^u a_z \, dz\Big) - \psi (s, y)  \Big].
\end{align*}
For every \(a \in \mathcal{A}^*\), applying It\^o's formula to 
\[
Y_u := e^{- \frac{1}{2} \int_0^u \|a_z\|^2 \, dz} \psi \Big(s + u, y + X_{u} + \sqrt{p - 1} \, \int_0^u a_z \, dz\Big), 
\]
we obtain that 
\begin{align*}
    d Y_u &= - \frac{1}{2} \|a_u\|^2 Y_u \, du + e^{- \frac{1}{2} \int_0^u \|a_z\|^2 \, dz} \Big( \left.\frac{\partial \psi }{\partial t}\right|_{t=s+u} \Big(s + u, y + X_{u} + \sqrt{p - 1} \, \int_0^u a_z \, dz\Big) 
    \\&\hspace{0.5cm} + \Big\langle D \psi \Big(s + u, y + X_{u} + \sqrt{p - 1} \, \int_0^u a_z \, dz\Big), \sqrt{p - 1} \, a_u \Big\rangle 
    \\&\hspace{0.5cm} + \frac{1}{2} \on{tr} \Big[ D^2 \psi \Big(s + u, y + X_{u} + \sqrt{p - 1} \, \int_0^u a_z \, dz\Big) \Big] \Big) \, du 
    \\&\hspace{0.5cm} + e^{- \frac{1}{2} \int_0^u \|a_z\|^2 \, dz} \Big \langle D \psi \Big(s + u, y + X_{u} + \sqrt{p - 1} \, \int_0^u a_z \, dz\Big), d X_u \Big\rangle.  
\end{align*}
The \(dX_u\)-integral has bounded quadratic variation and therefore is a true martingale with zero expectation. Using this fact, we find that
\begin{align*}
        0 &\leq \sup_{a \in \mathcal{A}^*} \E^\W \Big[ \int_0^u e^{- \frac{1}{2} \int_0^z \|a_r\|^2 \, dr} \, L \Big(s + z, y + X_z + \sqrt{p-1}\, \int_0^z a_r \, dr \Big) \, dz \Big], 
\end{align*}
where, recalling that \(\psi\) is bounded away from zero, 
\begin{align*}
L (z, x) &:=  \frac{\partial \psi}{\partial t} (z, x) + \sup_{a \in \bR^d} \Big[ \sqrt{p - 1} \langle D \psi (z, x), a \rangle - \frac{1}{2} \psi (z, x) \| a \|^2 \Big] + \frac{1}{2} \on{tr} \big[ D^2 \psi (z, x)\big]
\\&= \frac{\partial \psi}{\partial t} (z, x) + \frac{1}{2} (p - 1) \frac{\|D \psi (z, x)\|^2}{ \psi (z, x)} + \frac{1}{2} \on{tr} \big[ D^2 \psi (z, x)\big].
\end{align*}
Using that \(\psi \in C^\infty_b ([0, t] \times \bR^d)\) is Lipschitz continuous with Lipschitz derivatives, we further obtain that 
\begin{align*}
        0 &\leq \sup_{a \in \mathcal{A}^*} \E^\W \Big[ \int_0^u e^{- \frac{1}{2} \int_0^z \|a_r\|^2 \, dr} \, \Big( L \Big(s + z, y + X_z + \sqrt{p - 1}\, \int_0^z a_r \, dr \Big) - L (s, y) \Big) \, dz \Big] 
        \\&\hspace{1cm}+ \sup_{a \in \mathcal{A}^*} \E^\W \Big[ L (s, y) \int_0^u e^{- \frac{1}{2} \int_0^z \|a_r\|^2 \, dr} \, dz \Big] 
        \\&\leq \text{const.} \ \sup_{a \in \mathcal{A}^*} \E^\W \Big[ \int_0^u e^{- \frac{1}{2} \int_0^z \|a_r\|^2 \, dr} \, \Big( z + \|X_z\| + \sqrt{z (p - 1)}\,\, \Big(\int_0^z \|a_r\|^2 \, dr\Big)^{1/2} \, \Big)  \, dz \Big]   \\&\hspace{1cm}+ \sup_{a \in \mathcal{A}^*} \E^\W \Big[ L (s, y) \int_0^u e^{- \frac{1}{2} \int_0^z \|a_r\|^2 \, dr} \, dz \Big] 
        \\&\leq \text{const.} \ \Big( u^2 + u^{3/2} \Big)  + \sup_{a \in \mathcal{A}^*} \E^\W \Big[ L (s, y) \int_0^u e^{- \frac{1}{2} \int_0^z \|a_r\|^2 \, dr} \, dz \Big], 
\end{align*}
where the constant depends on \(\psi\) and \(p\).
Finally, dividing through \(u\) and letting \(u \downarrow 0\) yields that 
\begin{align*}
    0 \leq \limsup_{u \downarrow 0} \frac{1}{u} \int_0^u \sup_{a \in \mathcal{A}^*} L (s, y) \, \E^\W \Big[ e^{- \frac{1}{2} \int_0^z \|a_r\|^2 \, dr} \Big] \, dz.
\end{align*}
In case \(L (s, y) < 0\), we have 
\[
\sup_{a \in \mathcal{A}^*} L (s, y) \, \E^\W \Big[ e^{- \frac{1}{2} \int_0^z \|a_r\|^2 \, dr} \Big] = L (s, y) \, \inf_{a \in \mathcal{A}^*} \E^\W \Big[ e^{- \frac{1}{2} \int_0^z \|a_r\|^2 \, dr} \Big] \leq \frac{L (s, y) \inf f }{2 \sup f}, 
\]
which yields the contradiction
\[
\limsup_{u \downarrow 0} \frac{1}{u} \int_0^u \sup_{a \in \mathcal{A}^*} L (s, y) \, \E^\W \Big[ e^{- \frac{1}{2} \int_0^z \|a_r\|^2 \, dr} \Big] \, dz \leq \frac{L (s, y) \inf f }{2 \sup f} < 0.
\]
Thus, we must have \(L (s, y) \geq 0\), which proves that \(\widehat{v}\) has the subsolution property. 

Let us now discuss the supersolution property, leaving out some computations for brevity. Take \((s, y) \in [0, t) \times \bR^d\) and a function \(\psi \in C^\infty_b ([0, t] \times \bR^d)\) such that \(\widehat{v} \geq \psi\) and \(\psi (s, y) = \widehat{v} (s, y)\). As \(\widehat{v} \geq \inf f > 0\), we may assume that \(\psi \geq \inf f / 2\).
Let \(a^* \in \mathcal{A}\) be the constant strategy
\[
a^* \equiv \frac{\sqrt{p - 1}}{\psi (s, y)} \, D \psi (s, y).
\]
From \eqref{eq: DPP application} and an It\^o formula computation, we get that 
\begin{align*}
    0 &\geq \sup_{a \in \mathcal{A}} \E^\W \Big[ e^{- \frac{1}{2} \int_0^{u} \|a_z\|^2 \, dz} \, \psi \Big(s + u, y + X_{u} + \sqrt{p - 1} \, \int_0^u a_z \, dz\Big) - \psi (s, y) \Big]
    \\&\geq \E^\W \Big[ e^{- \frac{1}{2} \int_0^{u} \|a^*_z\|^2 \, dz} \, \psi \Big(s + u, y + X_{u} + \sqrt{p - 1} \, \int_0^u a^*_z \, dz\Big) - \psi (s, y) \Big]
    \\&= \E^\W \Big[ \int_0^u e^{- \frac{1}{2} \int_0^z \|a^*_r\|^2 \, dr} K \Big(s + z, y + X_z + \sqrt{p - 1} \, \int_0^z a^*_r \, dr \Big) \, dz \Big], 
\end{align*}
where
\begin{align*}
    K (z, x) := \frac{\partial \psi}{\partial t} (z, x) &+ \frac{p - 1}{\psi (s, y)} \, \langle D \psi (z, x), D \psi (s, y) \rangle + \frac{1}{2} \on{tr} \big[ D^2 \psi (z, x) \big] \\&- \frac{p - 1}{2 \psi (s, y)^2} \| D \psi (s, y) \|^2 \psi (z, x).
\end{align*}
Dividing by \(u\), letting \(u \downarrow 0\), and using the dominated convergence theorem (that is applicable due to the fact that \(\psi \in C^\infty_b ([0, t] \times \bR^d)\)), we finally find that 
\begin{align*}
    0 &\geq \liminf_{u \downarrow 0} \E^{\W} \Big[ \frac{1}{u}  \int_0^u e^{- \frac{1}{2} \int_0^z \|a^*_r\|^2 \, dr} K \Big(s + z, y + X_z + \sqrt{p - 1} \,\int_0^z a^*_r \, dr \Big) \, dz \Big] 
    \\&= K (s, y) = L (s, y). 
\end{align*}
This proves the supersolution property.

In summary, we conclude that \(\widehat{v}\) is a viscosity solution to the nonlinear PDE \eqref{eq: viscosity}. Notice also that \(\widehat{v} (t, x) = f(x) = v (t, x)\) for all \(x \in \bR^d\).
Thus, both \(v\) and \(\widehat{v}\) can be transferred to the heat equation with the same terminal data via the change of variable \(y \mapsto y^p\). We conclude from its uniqueness (among bounded viscosity solutions; see, e.g., \cite[Theorem~3.6]{nisio_book} for a suitable uniqueness result) that \(v = \widehat{v}\), which entails (b). 
\qed

\appendix 

\section{A Dynamic Programming Principle}

The purpose of this section is to establish a dynamic programming principle that is used in the main body of the proof for Theorem~\ref{theo-VF}. We work in the framework introduced in Section~\ref{sec: setup}, i.e., \(\Omega\) is the canonical space \(C ([0, T]; \bR^d)\), \(\cF\) is the corresponding Borel \(\sigma\)-field, \(\W\) is the Wiener measure, and \((\cF^0_t)_{t \in [0, T]}\) is the natural filtration generated by the coordinate process \(X_t (\omega) = \omega (t)\). Moreover, with slight abuse of notation, \(\mathcal{A}\) denotes the set of all \(\bR^d\)-valued \((\cF^0_t)_{t \in [0, T]}\)-progressively measurable processes \(a \colon [0, T] \times \Omega \to \bR^d\) such that \(\int_0^T \|a_s (\omega)\|^2 \, ds < \infty\) for all~\(\omega \in \Omega\). 

For \(\omega, \alpha \in \Omega\) and \(t \in [0, T]\), we define
\[
( \omega \otimes_t \alpha ) (s):= \begin{cases} \omega (s), & s \leq t, \\ \omega (t) + \alpha (s) - \alpha (t), & s > t. \end{cases} 
\] 
Building on this notation, for a fixed \(p \geq 1\), \(\varphi \in \Li\) and \((t, \omega) \in [0, T] \times \Omega\), we define
\[
J (t, \omega, a, \varphi) := \E^\W\Big[ e^{- \frac{1}{2} \int_t^{T} \| a_s (\omega \, \otimes_t \, X) \|^2 \, ds} \varphi \Big ( \omega \otimes_t X + \sqrt{p - 1} \int_t^{\cdot \vee t} a_s (\omega \otimes_t X) \, ds \Big) \Big], 
	\]
and 
\[
\widehat{\cE}_t (\varphi) (\omega) := \sup_{a \in \mathcal{A}} J (t, \omega, a, \varphi).
\] 
The control class \(\mathcal{A}\) differs slightly from the one used in the main body of this paper, where we use \((\cF^\W_t)_{t \in [0, T]}\) instead of \((\cF^0_t)_{t \in [0, T]}\) for the progressive measurability, and require square-integrability only \(\W\)-almost surely. This restriction causes no problem for the proof of Theorem~\ref{theo-VF}, as the following lemma explains. We denote by \(\mathcal{A}^\W\) the set of all \(\bR^d\)-valued \((\cF^\W_t)_{t \in [0, T]}\)-progressively measurable processes such that \(\W\)-a.s. \(\int_0^T \|a_s\|^2 \, ds < \infty\).
\begin{lemma} \label{lem: coincide}
    For all \(\varphi \in \Li\), 
    \[
    \sup_{a \in \mathcal{A}} J (0, 0, a, \varphi) = \sup_{a \in \mathcal{A}^\W} J (0, 0, a, \varphi),
    \]
    where \(0\) refers to the zero path.
\end{lemma}
\begin{proof}
Endow \(\mathcal{A}^\W\) with the topology of \(L^2\)-convergence in \(\W\)-probability, i.e., \(a^n \to a\) if
\[
\W \Big( \int_0^T \| a^n_s - a_s \|^2 \, ds \geq \varepsilon \Big) \to 0, \quad n \to \infty, \quad \text{for all } \varepsilon > 0. 
\]
By the dominated convergence theorem, the map \(a \mapsto J (0, 0, a, \varphi)\) is continuous from \(\mathcal{A}^\W\) to \(\bR\). As the set \(\mathcal{S}^\W\) of bounded \((\cF^\W_t)_{t \in [0, T]}\)-adapted simple control processes is dense in \(\mathcal{A}^\W\) with this topology (see, e.g., \cite[Proposition~3.2.6]{KaratzasShreve1991}), we find that 
\begin{align*}
    \sup_{a \in \mathcal{A}^\W} J (0, 0, a, \varphi) = \sup_{a \in \mathcal{S}^\W} J (0, 0, a, \varphi).
\end{align*}
Finally, by \cite[Lemma~2.1]{STZ_11}, for any control process \(a\in \mathcal{S}^\W\) there exists a control process \(a' \in \mathcal{A}\) such that \(a = a'\) \(dt \otimes \W\)-almost everywhere. Here, we used that any \(a \in \mathbb{S}^\W\) is bounded and therefore, we may take \(a'\) bounded as well.
This yields that 
\[
\sup_{a \in \mathcal{S}^\W} J (0, 0, a, \varphi) \leq \sup_{a \in \mathcal{A}} J (0, 0, a, \varphi).
\]
Noting the trivial inequality 
\[
    \sup_{a \in \mathcal{A}} J (0, 0, a, \varphi) \leq \sup_{a \in \mathcal{A}^\W} J (0, 0, a, \varphi), 
    \]
the proof is complete.
\end{proof}

\begin{lemma} \label{lem: Lip}
	Let \(t \in [0, T]\) and \(\varphi \in \Li\) with Lipschitz constant \(L > 0\). Then, for every \(a \in \mathcal{A}\) and \(\omega, \alpha \in \Omega\), with \(a^{t, \omega} := a (\omega \, \otimes_t X)\), 
	\[
	| J (t, \omega, a, \varphi) - J (t, \alpha, a^{t, \omega}, \varphi) | \leq L \| \omega - \alpha \|_t.
	\] 
	Moreover, \(\widehat{\cE}_t (\Li) \subset \Li\).
\end{lemma}
\begin{proof}
The first claim follows immediately from the fact that 
\[
\| \omega \otimes_t X - \alpha \otimes_t X \|_T \leq \| \omega - \alpha \|_t. 
\]
For the second, notice that 
\[
\widehat{\cE}_t (\varphi) (\omega) - \widehat{\cE}_t (\varphi) (\alpha) \leq \sup_{a \in \mathcal{A}} |J (t, \omega, a, \varphi) - J (t, \alpha, a^{t, \omega}, \varphi) | \leq L \| \omega - \alpha \|_t.
\] 
By symmetry, this implies the claim. 
\end{proof} 

The final part of Lemma~\ref{lem: Lip} implies that \(\omega \mapsto \widehat{\cE}_t (\varphi) (\omega)\) is \(\cF\)-measurable for all \(t \in [0, T]\) and \(\varphi \in \Li\). As any \(\widehat{\cE}_t (\varphi) (\omega)\) depends on \(\omega\) only through \((\omega (s))_{s \leq t}\), following from its definition, Galmarino's test further shows that \(\omega \mapsto \widehat{\cE}_t (\varphi) (\omega)\) is even \(\cF^0_t\)-measurable. 
The following is the dynamic programming principle (or time-consistency property) for our framework. Its proof adapts well-known arguments for establishing dynamic programming in settings with sufficiently regular value functions, see, e.g., \cite[Theorem~4.3.3]{YZ_99}.

\begin{theorem} \label{theo: semigroup}
For every \(\omega \in \Omega, t \in [0, T), s \in (t, T]\) and \(\varphi \in \Li\),  
	\[
	\widehat{\cE}_t (\widehat{\cE}_s (\varphi)) (\omega) = \widehat{\cE}_t (\varphi) (\omega).
	\] 
\end{theorem}

\begin{proof}
	Throughout the proof, take \(\omega \in \Omega, t \in [0, T), s \in (t, T]\) and \(\varphi \in \Li\).
	We prove the two inequalities \(\leq\) and \(\geq\), starting with the latter.
To ease the presentation, set 
\begin{align*}
D^a_{r, u} &:= e^{- \frac{1}{2} \int_r^u \| a_v\|^2 \, dv}, 
\\
H_{s, t}^a (\alpha) &:= \alpha + \sqrt{p - 1} \int_t^{(\cdot \, \wedge s) \vee t} a_r (\alpha) \, dr.
\end{align*} 
Using that \( (X_{u} - X_s)_{u \in [s, T]}\) is \(\W\)-independent of \(\cF^0_s\), we find that 
\begin{align*}
	\widehat{\cE}_t (\varphi) (\omega) &= \sup_{a' \in \mathcal{A}} \, \int_\Omega D_{t, s}^{a'} (\omega \otimes_t \alpha) \E^\W\Big[ D_{s, T}^{a'} ((\omega \otimes_t \alpha) \otimes_s X) \\&\hspace{2cm} \varphi \Big( (\omega \otimes_t \alpha) \otimes_s X + \sqrt{p - 1} \int_t^{\cdot \vee t} a_r' ( (\omega \otimes_t \alpha) \otimes_s X) \, dr \Big) \mid \cF_s^0 \Big] (\alpha) \, \W (d \alpha)
	\\&= \sup_{a' \in \mathcal{A}} \, \int_\Omega D_{t, s}^{a'} (\omega \otimes_t \alpha) \E^\W\Big[ D_{s, T}^{a'} ((\omega \otimes_t \alpha) \otimes_s X) \\&\hspace{2cm} \varphi \Big( (\omega \otimes_t \alpha) \otimes_s X + \sqrt{p - 1} \int_t^{\cdot \vee t} a_r' ( (\omega \otimes_t \alpha) \otimes_s X) \, dr \Big) \Big] \, \W (d \alpha)
    \\&= \sup_{a' \in \mathcal{A}} \E^\W \Big[ D_{t, s}^{a'} (\omega \otimes_t X) \, J (s, H^{a'}_{s, t} (\omega \otimes_t X) , a' ( (\omega\otimes_t X) \otimes_s \cdot \, ), \varphi) \Big] 
	\\&\leq \sup_{a' \in \mathcal{A}} \E^\W\Big[ D_{t, s}^{a'} (\omega \otimes_t X) \widehat{\cE}_s (\varphi) (H^{a'}_{s, t} (\omega \otimes_t X)) \Big] 
	\\&= \sup_{a' \in \mathcal{A}} \E^\W\Big[ D_{t, T}^{a'} (\omega \otimes_t X) \widehat{\cE}_s (\varphi) (H^{a'}_{s, t} (\omega \otimes_t X)) \Big] 
	\\&= \widehat{\cE}_t (\widehat{\cE}_s (\varphi)) (\omega). \phantom \int
\end{align*}	
To pass from \(D^{a'}_{t, s}\) to \(D^{a'}_{t, T}\) in the second to last equality, we used the supremum, which allowed us to replace \(a'\) by \(a' \1_{[0, s]}\), as \(\widehat{\cE}_s (\varphi)\) is non-negative and \(\cF^0_s\)-measurable. 
This proves the \(\geq\) direction. 

For the converse direction, set \(\Omega_s := \{ \omega \in \Omega : \omega = \omega (\cdot \wedge s) \}\), which is a closed subspace of the Polish space \(\Omega\). Consequently, \(\Omega_s\) is separable. Let \((\omega^n)_{n = 1}^\infty\) be a countable dense subset of \(\Omega_s\). For each \(n, m \in \mathbb{N}\), let \(a^{n, m} \in \mathcal{A}\) be such that 
\[
\widehat{\cE}_s (\varphi) (\omega^n) \leq J (s, \omega^n, a^{n, m}, \varphi) + \frac{1}{m}.
\] 
Thanks to Lemma~\ref{lem: Lip}, we obtain
\begin{align*}
\widehat{\cE}_s (\varphi) (\omega) \leq \widehat{\cE}_s (\varphi) (\omega^n) + L \| \omega - \omega^n \|_s \leq  J (s, \omega, a^{n, m} (\omega^n \otimes_s X), \varphi) + \frac{1}{m} + 2 L \| \omega - \omega^n \|_s.
\end{align*} 
Denoting \((a^n)_{n = 1}^\infty\) an enumeration of \((a^{n, m} (\omega^n \otimes_s X))_{n, m = 1}^\infty\), we obtain that 
\begin{align} \label{eq: opti}
\widehat{\cE}_s (\varphi) (\omega) = \sup_{n \in \mathbb{N}} J (s, \omega, a^n, \varphi).
\end{align}
Take \(\varepsilon > 0\) and set 
\[
F_\varepsilon (\alpha) := \inf \big\{ n \in \mathbb{N} \colon J (s, \alpha, a^{n}, \varphi) \geq \widehat{\cE}_s (\varphi) (\alpha) - \varepsilon \big\}, \quad \alpha \in \Omega.
\]
It is easy to see that \(F_\varepsilon\) is \(\cF^0_s\)-measurable and that \(F_\varepsilon < \infty\) by \eqref{eq: opti}. Take an arbitrary \(a \in \mathcal{A}\) and set, for \(\alpha \in \Omega\),  
\[
a'_r (\alpha) := \begin{cases} a_r (\alpha), & r \leq s, \\ \sum_{n = 1}^\infty \1_{\{F_\varepsilon (H_{s, t}^a (\omega \otimes_t \alpha)) = n \}} a^n_r (\alpha), & r \in (s, T]. \end{cases}
\] 
Notice that \(a' \in \mathcal{A}\), as \(F_\varepsilon (H_{s, t}^a (\omega \otimes_t \alpha))\) is \(\cF^0_s\)-measurable. 
We are in the position to finish the proof. Using that \( (X_{u} - X_s)_{u \in [s, T]}\) is \(\W\)-independent of \(\cF_s^0\), we find that 
\begin{align*}
	\widehat{\cE}_t (\varphi) (\omega) \geq J (t, \omega, a', \varphi) 
&= \E^\W \Big[ D_{t, s}^{a'} (\omega \otimes_t X) \, J (s, H^{a'}_{s, t} (\omega \otimes_t X) , a' ( (\omega\otimes_t X) \otimes_s \cdot \, ), \varphi) \Big]
	\\&= \E^\W \Big[ D^{a}_{t, s} (\omega \otimes_t X) \sum_{n = 1}^\infty \1_{\{ F_\varepsilon (H_{s, t}^a (\omega \otimes_t X)) = n \}} J (s, H_{s, t}^a (\omega \otimes_t X), a^n, \varphi) \Big] 
	\\&\geq \E^\W \Big[ D^{a}_{t, s} (\omega \otimes_t X) \widehat{\cE}_s (\varphi) (H_{s, t}^a (\omega \otimes_t X))  \Big] - \varepsilon, 
\end{align*}
where we used that \(a^n ( \alpha \otimes_s \omega ) = a^n (\omega)\) by its construction.
Letting \(\varepsilon \downarrow 0\) and optimizing over all \(a \in \mathcal{A}\) shows that 
\begin{align*}
	\widehat{\cE}_t (\varphi) (\omega) &\geq \sup_{a \in \mathcal{A}} \E^\W \Big[ D^{a}_{t, s} (\omega \otimes_t X) \widehat{\cE}_s (\varphi) (H_{s, t}^a (\omega \otimes_t X))  \Big] 
	\\&= \sup_{a \in \mathcal{A}} \E^\W \Big[ D^{a}_{t, T} (\omega \otimes_t X) \widehat{\cE}_s (\varphi) (H_{s, t}^a (\omega \otimes_t X))  \Big] = \widehat{\cE}_t (\widehat{\cE}_s (\varphi)) (\omega), 
\end{align*}
which completes the proof. 
\end{proof}

The following lemma provides a time-homogeneity property of the convex expectation \(\widehat{\cE}\). 

\begin{lemma} \label{lem: homoge}
	For all \(\varphi \in \Li\), \(t \in [0, T]\) and \(\omega \in \Omega\), 
	\[
		\widehat{\cE}_{t} ( \varphi ) (\omega) = \widehat{\cE}_0 ( \varphi (\omega \, \overline{\otimes}_t \, X)  ) (0), 
	\] 
	where \(0\) refers to the constant path and
	\[
	( \omega \, \overline{\otimes}_t \, X )(s) := \begin{cases} \omega (s), & s \leq t, \\ \omega (t) + X_{s - t} - X_0, & s > t. \end{cases}
	\] 
\end{lemma}
\begin{proof}
The claim essentially follows from the formula  
\[
\omega \, \overline{\otimes}_t \, (X_{(\cdot\, + t) \wedge T} - X_t) = \omega \otimes_t X, 
\]
and the fact that under \(\W\) the shifted process \((X_{t + s} - X_t)_{s \in [0, T - t]}\) is a standard Brownian motion. We provide the details.
For \(a \in \mathcal{A}\), set 
\[
c^a_r (X) := \begin{cases} a_{r + t} (\omega \, \overline{\otimes}_t \, X), & r \in [0, T - t], \\ 0, & r \in (T - t, T], \end{cases}
\]
which is clearly also an element of \(\mathcal{A}\).
Now, we obtain that 
\begin{align*}
    J (0, 0, c^a, \varphi (\omega \, \overline{\otimes}_t \, X)) &= \E^\W \Big[ e^{- \frac{1}{2} \int_t^T \|a_s (\omega \, \overline{\otimes}_t \, X)\|^2 \, ds} \varphi \Big( \omega \, \overline{\otimes}_t \, X + \sqrt{p - 1}\int_t^{\cdot \, \vee t} a_s (\omega \, \overline{\otimes}_t \, X) \, ds \Big) \Big]
    \\&= \E^\W \Big[ e^{- \frac{1}{2} \int_t^T \|a_s (\omega \, \otimes_t \, X)\|^2 \, ds} \varphi \Big( \omega \otimes_t X + \sqrt{p - 1} \int_t^{\cdot \, \vee t} a_s (\omega \otimes_t X) \, ds \Big) \Big]
    \\&= J (t, \omega, a, \varphi). \phantom \int 
\end{align*}
Hence, we obtain
\[
\widehat{\cE}_{t} ( \varphi ) (\omega) \leq \widehat{\cE}_0 ( \varphi (\omega \, \overline{\otimes}_t \, X)  ) (0).
\]
We turn to the converse inequality. For \(a \in \mathcal{A}\), set
\[
b^a_r (X) := \begin{cases} 0, & r \in [0, t), \\ a_{r - t} ( X_{(\cdot + t)\wedge T} - X_t ), & r \in 
[t, T], \end{cases} 
\]
which is clearly also an element of \(\mathcal{A}\).
Now, we obtain that 
\begin{align*} 
\widehat{\cE}_0 ( \varphi (\omega \, \overline{\otimes}_t \, X)  ) (0) 
&= \sup_{a \in \mathcal{A}} \E^\W \Big[ e^{- \frac{1}{2} \int_0^T \|a_s\|^2 \, ds} \varphi \Big( \omega \, \overline{\otimes}_t \, X + \sqrt{p - 1}\int_0^{(\cdot \, - \, t) \vee 0} a_s \, ds \Big) \Big] 
\\&= \sup_{a \in \mathcal{A}} \E^\W \Big[ e^{- \frac{1}{2} \int_0^{T - t} \|a_s\|^2 \, ds} \varphi \Big( \omega \, \overline{\otimes}_t \, X + \sqrt{p - 1}\int_0^{(\cdot \, - \, t) \vee 0} a_s \, ds \Big) \Big] 
\\&= \sup_{a \in \mathcal{A}} \E^\W \Big[ e^{-\frac{1}{2} \int_t^T \|b^a_s (\omega\, \otimes_t\, X)\|^2 \, ds} \varphi \Big( \omega \otimes_t X + \sqrt{p - 1} \int_t^{\cdot \, \vee t} b^a_s (\omega \otimes_t X) \, ds \Big) \Big] 
\\&\leq \widehat{\cE}_t (\varphi) (\omega). \phantom \int
\end{align*} 
This establishes the remaining inequality, completing the proof.
\end{proof}

For \(t \in [0, T], x \in \bR^d\) and \(\varphi \in \textit{Lip}_b (\bR^d; \bR_+)\), set 
\[
\widehat{S}_t [\varphi] (x) := \widehat{\cE}_0 \big( \varphi (x + X_t) \big) (0), 
\]
where \(0\) refers to the zero path.

\begin{corollary} \label{coro: for proof}
	For all \(n \geq 2, 0 \leq t_1 < t_2 < \dots < t_n \leq T\) and \(\varphi \in \textit{Lip}_b (\bR^{nd}; \bR_+)\), it holds that 
	\[
	\widehat{\cE}_0 (\varphi (X_{t_1}, \dots, X_{t_n}) ) = \widehat{\cE}_0 \Big( \widehat{S}_{t_n - t_{n - 1}} [\varphi (x_1, \dots, x_{n - 1}, \cdot \,)] (x_{n - 1}) \, \big|_{x_1 \, =\, X_{t_1}, \dots, x_{n - 1} \, =\, X_{t_{n - 1}}} \Big).
	\]
\end{corollary} 

\begin{proof}
	The claim follows from Theorem~\ref{theo: semigroup} and Lemma~\ref{lem: homoge}.
\end{proof}

\section*{Declaration of AI use.}
ChatGPT was used to improve the exposition and to suggest refinements to some arguments in the applications. The main results and their proofs were developed by the authors. All AI-assisted material was verified by the authors, who take full responsibility for the content.

\end{document}